\documentclass[11pt]{article}

\usepackage[margin=1in]{geometry}
\usepackage{amsmath,amssymb,amsthm}
\usepackage{xcolor}
\usepackage{tikz}
\usepackage[hidelinks,unicode]{hyperref}
\hypersetup{pdftitle={0/1 integer programming over polytopes of Chv\'atal rank one is as hard as lattice problems},pdfauthor={Alberto Del Pia}}

\newtheorem{theorem}{Theorem}
\newtheorem{lemma}[theorem]{Lemma}
\newtheorem{proposition}[theorem]{Proposition}
\newtheorem{corollary}[theorem]{Corollary}
\newtheorem{question}{Question}
\newtheorem*{theoremA}{Theorem A}
\newtheorem*{theoremB}{Theorem B}
\newtheorem*{theoremC}{Theorem C}
\newtheorem*{theoremD}{Theorem D}
\theoremstyle{remark}
\newtheorem{remark}[theorem]{Remark}

\newcommand{\R}{\mathbb{R}}
\newcommand{\Z}{\mathbb{Z}}
\newcommand{\Q}{\mathbb{Q}}
\newcommand{\0}{\mathbf{0}}
\newcommand{\1}{\mathbf{1}}
\newcommand{\T}{^{\mathsf{T}}}
\newcommand{\conv}{\operatorname{conv}}
\newcommand{\poly}{\operatorname{poly}}
\newcommand{\fl}[1]{\left\lfloor #1 \right\rfloor}
\newcommand{\width}{\operatorname{width}}
\DeclareMathOperator{\LWE}{LWE}

\title{$0/1$ integer programming over polytopes of Chv\'atal rank one\\ is as hard as lattice problems}
\author{Alberto Del Pia\thanks{Department of Industrial and Systems Engineering \& Wisconsin Institute for Discovery, University of Wisconsin--Madison, Madison, WI, USA. E-mail: \texttt{delpia@wisc.edu}.}}
\date{October 23, 2026}

\begin{document}
\maketitle

\begin{abstract}
A rational polyhedron has Chv\'atal rank at most one if a single round of Chv\'atal--Gomory cuts yields its integer hull.
In 1973, Chv\'atal asked whether $0/1$ integer programs over polytopes of bounded Chv\'atal rank can be solved in polynomial time.
For polyhedra of bounded Chv\'atal rank, integer feasibility is in NP $\cap$ coNP, so it is unlikely to be NP-hard.
We answer Chv\'atal's question negatively already in its most restricted form: for Chv\'atal rank one, for integer feasibility, and even for the SAT polytopes of formulas whose clauses have at most three literals.
Our answer holds under either of two standard assumptions from lattice-based cryptography.
First, a polynomial-time algorithm for integer feasibility on such SAT polytopes of Chv\'atal rank one would solve bounded distance decoding with polynomial factors in deterministic polynomial time, contradicting the conjectured hardness of approximating lattice problems within polynomial factors.
Second, assuming the hardness of learning with errors, an average-case analog of bounded distance decoding, the problem is also hard on average, for a polynomial-time samplable distribution of formulas.
This also gives negative answers to related questions of Cornu\'ejols, Lee, and Li.
We first prove these results for parallelepipeds, where they rest on an elementary sufficient condition: a polyhedron has Chv\'atal rank at most one if its width is less than one along every row of some unimodular matrix.
For our parallelepipeds, such a matrix exists but is hard to find.
For SAT polytopes, we represent candidate lattice vectors by their residues modulo small integers, and binary arithmetic on additional $0/1$ variables checks that these vectors lie in the lattice.
\end{abstract}



\section{Introduction}\label{sec:intro}

Let $P\subseteq\R^n$ be a rational polyhedron and let $P_I=\conv(P\cap\Z^n)$ be its \emph{integer hull}.
A \emph{Chv\'atal--Gomory cut} (CG cut) for $P$ is an inequality $c\T x\le\fl{\delta}$, where $c\in\Z^n$ and $c\T x\le\delta$ is valid for $P$ \cite{Gomory58,Chvatal73}.
The \emph{Chv\'atal closure} $P'$ of $P$ is the set of vectors satisfying all CG cuts for $P$.
We have $P_I\subseteq P'\subseteq P$, and $P'$ is again a rational polyhedron \cite{Schrijver80}.
Iterating the closure operation yields $P_I$ after finitely many steps, as shown by Chv\'atal \cite{Chvatal73} for polytopes and by Schrijver \cite{Schrijver80} for rational polyhedra.
The \emph{Chv\'atal rank} of $P$ is the number of steps needed.

Polyhedra of Chv\'atal rank zero are the \emph{integral} polyhedra, and recognizing them is coNP-complete \cite{PapadimitriouYannakakis90}.
Polyhedra of Chv\'atal rank at most one include several classical examples from combinatorial optimization \cite{Schrijver03}: the fractional matching and edge cover polytopes and, more generally, the fractional $b$-matching and $b$-edge cover polyhedra of a graph, the fractional stable set polytope of a t-perfect graph, and the strong fractional stable set polytope of an h-perfect graph.
More recent examples arise from a generalization of bidirected graphs \cite{DelPiaMusitelliZambelli18}, from the unit cube \cite{CornuejolsLee18}, and from binary polynomial optimization \cite{DelPiaDiGregorio21}.
In general, recognizing polyhedra of Chv\'atal rank at most one is NP-hard \cite{CornuejolsLi18,CornuejolsLeeLi20}.

Deciding whether $P\cap\Z^n\neq\emptyset$, which we call \emph{integer feasibility}, is NP-complete \cite[Theorem~18.1]{Schrijver86}.
If $P$ is integral, it is easy: $P$ contains an integral vector if and only if $P\neq\emptyset$, which can be decided in polynomial time by linear programming \cite[Theorem~13.4]{Schrijver86}.
Chv\'atal observed in 1973 that, for each fixed $k$, integer programs over polytopes of Chv\'atal rank at most $k$ admit a good characterization \cite[Section~7]{Chvatal73}.
Boyd and Pulleyblank showed in 1984 that, for each such $k$, integer feasibility is in NP $\cap$ coNP for polyhedra of Chv\'atal rank at most $k$ \cite[Theorem~5.4]{BoydPulleyblank09} (see also \cite[Theorem~23.5 and Corollary~23.5a]{Schrijver86}).
Chv\'atal also asked whether, for each fixed $k$, $0/1$ integer programs over polytopes of Chv\'atal rank at most $k$ can be solved in polynomial time \cite[Section~7]{Chvatal73}.

We consider the most restricted case of his question: Chv\'atal rank at most one, that is, $P'=P_I$, and integer feasibility in place of optimization.
In this case, it is elementary that the problem is in NP $\cap$ coNP.

\begin{question}\label{q:cube}
Let $P=\{x\in[0,1]^n: Cx\le d\}$ with $C$ and $d$ rational, and suppose $P'=P_I$.
Can we decide whether $P$ contains an integral vector in time polynomial in the encoding size of $(C,d)$?
\end{question}

Question~\ref{q:cube} is at the heart of this paper.
Cornu\'ejols and Li stated the same question for general rational polyhedra \cite{CornuejolsLi16,CornuejolsLi18}, and it is the main motivation of the systematic study of this class by Cornu\'ejols, Lee, and Li \cite{CornuejolsLeeLi20}.

\begin{question}\label{q:cll}
Let $P=\{x\in\R^n: Cx\le d\}$ with $C$ and $d$ rational, and suppose $P'=P_I$.
Can we decide whether $P$ contains an integral vector in time polynomial in the encoding size of $(C,d)$?
\end{question}

Cornu\'ejols, Lee, and Li state Question~\ref{q:cll} for polyhedra of Chv\'atal rank exactly one, and, for both questions, the versions for rank exactly one are equivalent to those above (Lemma~\ref{lem:exactlyone}).

We regard Questions~\ref{q:cube} and~\ref{q:cll} as \emph{promise problems}, with the promise $P'=P_I$, as in \cite{CornuejolsLi18,CornuejolsLeeLi20}.
An algorithm solves such a problem in polynomial time if it runs in polynomial time on all inputs and answers correctly on those that satisfy the promise, while its answers on other inputs may be arbitrary.

The promise matters.
Under the promise, $P\cap\Z^n=\emptyset$ if and only if $P'=\emptyset$.
Without it, deciding whether $P'=\emptyset$ is NP-complete, even for polyhedra containing no integral vector \cite{CornuejolsLi16,CornuejolsLeeLi20}, and separating over $P'$ is NP-hard \cite{Eisenbrand99}.
With the promise, NP-hardness is unlikely: since these promise problems are in NP $\cap$ coNP, a polynomial-time reduction from an NP-hard problem all of whose queries satisfy the promise would imply NP $=$ coNP \cite{GrollmannSelman88} (see also \cite[Theorem~7]{Goldreich06}).

Several problems in NP $\cap$ coNP were eventually shown to be polynomial, such as linear programming \cite[Theorem~13.4]{Schrijver86} and primality testing \cite{AgrawalKayalSaxena04}.
Yet no algorithm is known to benefit from the promise $P'=P_I$.
For instance, Cornu\'ejols, Lee, and Li prove a flatness theorem for compact convex sets with empty Chv\'atal closure, with an essentially tight bound, and observe that it yields a Lenstra-type algorithm \cite{Lenstra83} with running time $2^{O(n)}n^n$ times a polynomial in the encoding size \cite[Theorem~14, Proposition~15, and Section~3.3]{CornuejolsLeeLi20}.
This is slower than the randomized algorithm of Reis and Rothvoss for general integer programs, with running time $(\log n)^{O(n)}$ times a polynomial in the encoding size \cite{ReisRothvoss23}.
Similarly, branch-and-bound with variable branching does not benefit from the promise: it can require exponentially many nodes even on polytopes with empty Chv\'atal closure \cite{Jeroslow74}.

\paragraph{Our results.}
We first answer Question~\ref{q:cll} negatively, already for parallelepipeds, under two standard assumptions from lattice-based cryptography: that lattice problems such as the shortest vector problem are hard to approximate within polynomial factors, and that learning with errors, an average-case lattice problem, is hard.
Lattice problems are a natural source of hardness here: like integer feasibility under the promise $P'=P_I$, they are not expected to be NP-hard, since the decision versions of approximating the norm of a shortest nonzero vector, or the distance of a target vector from the lattice, within a factor $c\sqrt n$ for a suitable constant $c$, are in NP $\cap$ coNP \cite{AharonovRegev05}.
Under the same assumptions, we then answer Question~\ref{q:cube}, and with it Chv\'atal's question, negatively, even for SAT polytopes of formulas whose clauses have at most three literals (Theorems~C and~D).

Let $L\subseteq\R^n$ be a full-rank lattice and let $\lambda_1(L)$ be the Euclidean norm of a shortest nonzero vector of $L$.
In the \emph{bounded distance decoding problem} with factor $\alpha$ (BDD$_\alpha$) we are given a rational basis $B$ of $L$ and a rational target vector $t$ whose distance from $L$ is at most $\alpha\lambda_1(L)$, and we want to find a vector $v$ of $L$ closest to $t$, or equivalently the \emph{error} $t-v$.
In integer programming terms, we want to solve the integer least squares problem $\min\{\|Bx-t\|: x\in\Z^n\}$ under the promise that its optimal value is at most $\alpha\lambda_1(L)$.

\begin{theoremA}[see Proposition~\ref{cor:promise} and Theorem~\ref{thm:bdd}]
There is a deterministic polynomial-time reduction from BDD$_{1/(4n^2)}$ in dimension $n$ to integer feasibility on rational parallelepipeds in $\R^n$ with $P'=P_I$.
In particular, a polynomial-time algorithm for the latter gives a polynomial-time algorithm for BDD$_{1/(4n^2)}$.
\end{theoremA}

Combined with reductions of Lyubashevsky and Micciancio \cite{LyubashevskyMicciancio09}, Theorem~A has two further consequences (Corollary~\ref{cor:usvp}).
First, a polynomial-time algorithm for integer feasibility under the promise $P'=P_I$ would give a polynomial-time algorithm for uSVP$_{4n^2}$, that is, for finding a shortest nonzero vector of $L$ when every vector of $L$ of norm at most $4n^2\lambda_1(L)$ is an integer multiple of a shortest one.
Second, it would give a randomized polynomial-time algorithm for GapSVP$_\gamma$ with $\gamma=4n^{5/2}/\sqrt{\log n}$, that is, for approximating $\lambda_1(L)$ within a factor $\gamma$.

The best known algorithms for BDD, uSVP, and GapSVP with polynomial factors run in time $2^{O(n)}$ \cite{MicciancioVoulgaris13,ADRS15}.
It is conjectured that no polynomial-time algorithm, even a randomized or quantum one, approximates lattice problems such as the shortest vector problem within polynomial factors \cite[Conjectures~1 and~2]{MicciancioRegev09}.
This conjecture underlies lattice-based cryptography \cite{Regev09,Peikert16}.
Since GapSVP is the decision version of the shortest vector problem, Corollary~\ref{cor:usvp} shows that a polynomial-time algorithm for integer feasibility under the promise $P'=P_I$, even a randomized one, would contradict this conjecture.

The reduction of Theorem~A preserves the dimension (Theorem~\ref{thm:bdd}), and so do those of Lyubashevsky and Micciancio.
Hence, even an algorithm for integer feasibility under the promise $P'=P_I$ running in time $2^{o(n)}$ times a polynomial in the encoding size would give algorithms with the same running time for BDD and uSVP, and a randomized one for GapSVP (Corollary~\ref{cor:usvp}).

Theorem~A is a worst-case result: it concerns algorithms that are correct on every instance, and it does not show that hard instances are common, nor how to generate them.
From \emph{learning with errors} (LWE) \cite{Regev09}, the average-case problem on which much of lattice-based cryptography is built, we obtain a distribution of hard instances that can be sampled in polynomial time.
In LWE with modulus $q$, we are given a uniformly random matrix $A\in\Z_q^{n\times m}$ and $t\equiv A\T s+e\pmod q$, where $s\in\Z_q^n$ is uniformly random and $e\in\Z^m$ is a short random error, and we want to find $s$, or to distinguish $t$ from a uniformly random vector (decision-LWE).
Thus $t$ is a target vector close to the lattice of vectors $y\in\Z^m$ with $y\equiv A\T s\pmod q$ for some $s$, and LWE can be viewed as an average-case version of BDD.

\begin{theoremB}[see Proposition~\ref{cor:promise} and Theorem~\ref{cor:lwe}]
For every sufficiently large $n$, there is a polynomial-time samplable distribution of parallelepipeds in dimension $O(n\log n)$, given by integral constraints with coefficients bounded by a polynomial in $n$, that satisfy $P'=P_I$ with high probability.
A randomized polynomial-time algorithm that decides integer feasibility correctly with probability at least $2/3$ on these distributions gives one for decision-LWE in dimension $n$ with a polynomial modulus and Gaussian errors of inverse-polynomial rate.
\end{theoremB}

Combined with a result of Regev \cite{Regev09}, Theorem~B has a further consequence (Corollary~\ref{cor:regev}): a randomized polynomial-time algorithm that decides integer feasibility correctly with probability at least $2/3$ on these distributions would give a polynomial-time quantum algorithm for GapSVP$_\gamma$ and SIVP$_\gamma$ on $n$-dimensional lattices with $\gamma=\tilde O(n^3)$, where $\tilde O$ hides polylogarithmic factors.
Here SIVP$_\gamma$ asks for $n$ linearly independent vectors of the lattice whose largest norm is at most $\gamma$ times the smallest possible.
Such a quantum algorithm would contradict the conjecture that no polynomial-time quantum algorithm approximates lattice problems within polynomial factors \cite[Conjecture~2]{MicciancioRegev09}.

Theorems~A and~B hold even under a stronger promise, that every rational subpolyhedron of $P$ has Chv\'atal rank at most one (Proposition~\ref{cor:promise} and Corollary~\ref{cor:usvp}).
Unlike $P'=P_I$, this promise is preserved by branching and by adding cutting planes, since both replace $P$ by a rational subpolyhedron.
It is strictly stronger than $P'=P_I$, even for $0/1$ polytopes (Remark~\ref{rem:croppedcube}).

Theorems~A and~B are incomparable.
The negative answer from Theorem~A relies on the conjecture only for classical algorithms, and the reduction of Theorem~A preserves the dimension.
The negative answer that we derive from Theorem~B relies on the conjecture for quantum algorithms, and the reduction of Theorem~B increases the dimension from $n$ to $O(n\log n)$.
However, Theorem~B gives hard instances on average, with integer coefficients bounded by a polynomial in the dimension, so it also rules out pseudo-polynomial algorithms.
Its distribution also gives test instances for integer programming solvers.
With high probability, they satisfy the promise, and their answers are known to the generator: those in which $t$ comes from LWE contain an integral vector that it can compute, and those in which $t$ is uniformly random contain none.

\paragraph{The idea.}
The proofs of Theorems~A and~B rest on an elementary sufficient condition for Chv\'atal rank at most one (Lemma~\ref{lem:slab}): if $P$ has width less than one along the rows of a unimodular matrix $U$, then $P'=P_I$, and $P_I$ is obtained by adding at most $2n$ CG cuts.
We apply it to the rational parallelepipeds $\{x\in\R^n:\|Bx-t\|_\infty\le r\}$ with $r>0$, where $B$ is a basis of a lattice $L$ and $t$ is a target vector (Figure~\ref{fig:parallelepiped}).
Their integral vectors $x$ correspond, via $x\mapsto Bx$, to the vectors of $L$ within $\ell_\infty$-distance $r$ of $t$.
The unimodular matrices are exactly the matrices with rows $w_1\T B,\dots,w_n\T B$, where $w_1,\dots,w_n$ is a basis of the dual lattice $L^*=\{w\in\R^n: w\T y\in\Z\text{ for all }y\in L\}$, and the width along $B\T w_i$ is $2r\|w_i\|_1$.
The dual lattice $L^*$ has a basis with $\ell_1$-norms at most $n^2/\lambda_1(L)$, by a transference theorem of Banaszczyk \cite[Theorem~2.1]{Banaszczyk93} (Lemma~\ref{lem:dualbasis}).
Hence, for every $t$, these parallelepipeds and all their rational subpolyhedra have Chv\'atal rank at most one whenever $r<\lambda_1(L)/(2n^2)$, since widths can only decrease on subpolyhedra.
We call such parallelepipeds \emph{thin}.

For Theorem~A, given an instance $(B,t)$ of BDD$_{1/(4n^2)}$, we try radii $r$ that are powers of two, and one of them gives a thin parallelepiped, which satisfies the promise by Proposition~\ref{cor:promise} and contains the integral vector corresponding to the closest vector.
We then find its entries bit by bit, with queries for parallelepipeds defined by sublattices of $L$, which are thin as well.

For Theorem~B, given an instance $(A,t)$ of decision-LWE, the lattice of the vectors $y\in\Z^m$ with $y\equiv A\T s\pmod q$ for some $s$ depends only on $A$, which is uniformly random, so with high probability it has no short nonzero vectors (Lemma~\ref{lem:lambda1}).
Hence $(A,t)$ yields a single parallelepiped that is thin, and so satisfies the promise, with high probability.
If $t$ comes from LWE, then with high probability the error is short, and so this parallelepiped contains an integral vector.
If $t$ is uniformly random, the parallelepiped contains none with high probability, by a counting argument.
Thus a single call to an algorithm for integer feasibility distinguishes the two cases.

\paragraph{Easy with a short dual basis, hard without.}
For thin parallelepipeds, the promise holds because the dual lattice has a basis with small $\ell_1$-norms, and knowing such a basis makes the problem easy.
With such a basis, Babai's round-off algorithm \cite{Babai86} finds the integral vector if there is one, and otherwise the basis gives a \emph{refutation}, that is, a set of CG cuts with empty intersection (Remark~\ref{rem:refutation}).
In the variables $Ux$, where $U$ has rows $w_1\T B,\dots,w_n\T B$, the parallelepiped has width less than one along every coordinate direction, so each variable takes at most one integer value and branch-and-bound solves it with $O(n)$ nodes.
This parallels the integer programs of Aardal, Hurkens, and Lenstra \cite{AardalHurkensLenstra00} and of Pataki, Tural, and Wong \cite{PatakiTuralWong10}, which are difficult for branch-and-bound in the original variables but much easier after a reformulation computed by basis reduction.
Here, however, finding such a basis is as hard as BDD$_{1/(4n^2)}$ for the parallelepipeds of Theorem~A, and as hard as decision-LWE for those of Theorem~B.

In proof complexity terms, CG cuts form the Cutting Planes proof system \cite{CookCoullardTuran87}, and finding a refutation of size polynomial in that of a shortest one is NP-hard in general \cite{GKMP20}.
Under the promise $P'=P_I$, a rational polyhedron without integral vectors has $P'=\emptyset$.
Since $P'$ is the intersection of finitely many CG cuts \cite{Schrijver80}, Helly's theorem then gives a refutation with at most $n+1$ CG cuts.
For our parallelepipeds, however, finding any refutation is as hard as BDD$_{1/(4n^2)}$ and decision-LWE (Remark~\ref{rem:refutation}).

\paragraph{Back to the unit cube.}
We now return to Question~\ref{q:cube}, through SAT polytopes.
For a formula $F$ in conjunctive normal form on variables $x_1,\dots,x_N$, the \emph{SAT polytope} $P(F)$ is the set of vectors $x\in[0,1]^N$ that satisfy $\sum_{i\in I}x_i+\sum_{i\in J}(1-x_i)\ge1$ for every clause of $F$, where $I$ and $J$ are the sets of indices of the variables that appear in the clause unnegated and negated, respectively.
Its integral vectors are the satisfying assignments of $F$.
Cornu\'ejols, Lee, and Li also asked whether satisfiability can be decided in polynomial time for formulas whose SAT polytope has Chv\'atal rank exactly one \cite[Open Question~2]{CornuejolsLeeLi20}.
Up to the equivalence between rank exactly one and rank at most one (Lemma~\ref{lem:exactlyone}), this is the special case of Question~\ref{q:cube} for SAT polytopes.
Under the promise $P(F)'=P(F)_I$, the problem is easy if every clause has at most two literals, since then satisfiability is polynomial even without the promise \cite{AspvallPlassTarjan79}.
In this case, the Chv\'atal closure of the SAT polytope is also well understood \cite{GerardsSchrijver86}.
The problem is also easy if every clause has at least three literals, since then the promise implies that $F$ is satisfiable \cite[Remark~3]{CornuejolsLeeLi20}.
Cornu\'ejols, Lee, and Li point out that the remaining case consists of formulas with both clauses of two literals and clauses of at least three literals \cite[Section~2.1]{CornuejolsLeeLi20}.
We show that Theorems~A and~B extend, with different parameters, to formulas whose clauses have at most three literals, and the formulas in our reductions are of this remaining kind.

\begin{theoremC}[see Proposition~\ref{prop:satlattice} and Theorem~\ref{thm:satbdd}]
There is a deterministic polynomial-time reduction from BDD$_{1/(4n^{5/2})}$ in dimension $n$ to satisfiability of formulas whose clauses have at most three literals and whose SAT polytopes satisfy $P(F)'=P(F)_I$.
In particular, a polynomial-time algorithm for the latter gives a polynomial-time algorithm for BDD$_{1/(4n^{5/2})}$.
\end{theoremC}

\begin{theoremD}[see Proposition~\ref{prop:satlattice} and Theorem~\ref{thm:satlwe}]
For every sufficiently large $n$, there is a polynomial-time samplable distribution of formulas whose clauses have at most three literals, with $O(n^2\log^4n)$ variables, whose SAT polytopes satisfy $P(F)'=P(F)_I$ with high probability.
A randomized polynomial-time algorithm that decides satisfiability correctly with probability at least $2/3$ on these distributions gives one for decision-LWE in dimension $n$ with a polynomial modulus and Gaussian errors of inverse-polynomial rate.
\end{theoremD}

Theorems~C and~D answer the question of Cornu\'ejols, Lee, and Li on SAT polytopes negatively.
Since SAT polytopes lie in the unit cube, they also answer Question~\ref{q:cube} negatively, and with it Chv\'atal's question.
Thus Chv\'atal's question has a negative answer already in its most restricted form: for Chv\'atal rank one, for integer feasibility, and even for SAT polytopes.

Like Theorems~A and~B, Theorems~C and~D hold even under the stronger promise that every rational subpolyhedron of $P(F)$ has Chv\'atal rank at most one (Proposition~\ref{prop:satlattice} and Corollary~\ref{cor:satusvp}).
They also hold for set partitioning polytopes with at most three ones in each row (Remark~\ref{rem:setpart}).

Theorems~C and~D are also incomparable.
The negative answer from Theorem~C relies on the conjecture only for classical algorithms.
As with Theorem~A, the reductions of Lyubashevsky and Micciancio show that a polynomial-time algorithm for satisfiability under the promise would give polynomial-time algorithms for uSVP$_{4n^{5/2}}$ and, with randomization, for GapSVP$_\gamma$ with $\gamma=4n^3/\sqrt{\log n}$ (Corollary~\ref{cor:satusvp}).
Unlike that of Theorem~A, the reduction of Theorem~C does not preserve the dimension, since the number of variables of its formulas depends on the encoding size of the input.
The negative answer that we derive from Theorem~D relies on the conjecture for quantum algorithms, through a result of Regev, with $\gamma=\tilde O(n^{5/2})$ (Corollary~\ref{cor:satregev}).
However, Theorem~D gives hard instances on average, and, as with Theorem~B, its distribution gives test instances with known answers, here for SAT solvers.

SAT polytopes are rational polytopes, so Theorems~C and~D also give hardness for general polytopes, but they do not imply Theorems~A and~B: Theorem~A has the better factor $1/(4n^2)$ and preserves the dimension, and the polytopes of Theorem~B have dimension $O(n\log n)$ instead of $O(n^2\log^4n)$.
In the other direction, Theorem~D has an advantage over Theorem~B: it gives the better factor $\tilde O(n^{5/2})$ for quantum algorithms, instead of $\tilde O(n^3)$.
Moreover, since SAT polytopes have coefficients in $\{0,\pm1\}$, even a pseudo-polynomial algorithm for them would contradict the conjecture for classical algorithms, through Theorem~C, whereas for parallelepipeds our argument uses Theorem~B, and so the conjecture for quantum algorithms.

\paragraph{The idea for SAT polytopes.}
The proofs of Theorems~C and~D need a different construction, since the promise for our SAT polytopes does not follow directly from Lemma~\ref{lem:slab}.
The satisfying assignments of our formulas correspond to the vectors of an integral lattice in a box around a target vector (Proposition~\ref{prop:satlattice}).
We represent the error, that is, the difference between the target and such a lattice vector, by the residues of its entries modulo small integers, each residue by a group of $0/1$ variables exactly one of which equals one.
Additional $0/1$ variables carry out, in binary, the arithmetic that recovers these entries from their residues and checks that the corresponding vector lies in the lattice.
The Chv\'atal closure of the SAT polytope lies in a face defined by equations that set a sum of terms $x_i$ and $1-x_i$ equal to one.
Each step of this arithmetic is modeled by the standard extended formulation of the convex hull of its finitely many $0/1$ solutions.
So the arithmetic holds as a linear identity on the whole face, not only at its integral vectors, and vectors of the dual lattice give integral directions along which this face has width less than one.
As in the proof of Lemma~\ref{lem:slab}, CG cuts along these directions fix all residues.
We use residues because the errors can be exponentially large, and residues modulo small integers represent them compactly, in a form that one round of CG cuts can fix.
As with the parallelepipeds, the vectors of the dual lattice that give these cuts exist but are hard to find (Remark~\ref{rem:findw}).

The paper is organized as follows.
Section~\ref{sec:prelim} contains preliminaries.
Section~\ref{sec:slab} proves our sufficient condition for Chv\'atal rank at most one (Lemma~\ref{lem:slab}) and applies it to parallelepipeds defined by lattices (Proposition~\ref{cor:promise}).
Section~\ref{sec:parhard} proves Theorems~A and~B and their consequences.
Section~\ref{sec:sat} constructs SAT polytopes defined by lattices, which play the role of the parallelepipeds (Proposition~\ref{prop:satlattice}), and Section~\ref{sec:sathard} proves Theorems~C and~D and their consequences.

\section{Preliminaries}\label{sec:prelim}

\paragraph{Notation.}
All vectors are column vectors.
We write $e_i$ for the $i$th unit vector, and $\0$ and $\1$ for the all-zeros and all-ones vectors.
All norms without subscript are Euclidean, and $\lfloor\cdot\rceil$ denotes rounding to a nearest integer.
The notation $\tilde O$ hides polylogarithmic factors.
A matrix $U\in\Z^{n\times n}$ is \emph{unimodular} if $\det U=\pm1$.

\paragraph{Polyhedra.}
CG cuts, the Chv\'atal closure $P'$, the integer hull $P_I$, and the Chv\'atal rank are defined in Section~\ref{sec:intro}.
For a polyhedron $P\subseteq\R^n$ and $c\in\R^n$, let $\delta_P(c)=\sup\{c\T x: x\in P\}$, and let $\width_P(c)=\delta_P(c)+\delta_P(-c)$ be the \emph{width} of $P$ along $c$.
A \emph{rational subpolyhedron} of $P$ is a rational polyhedron contained in $P$.

\paragraph{Lattices.}
A \emph{lattice} is a set $L=B\Z^n$, where $B\in\R^{n\times n}$ is nonsingular, and $B$ is a \emph{basis} of $L$.
For $i=1,\dots,n$, the $i$th \emph{successive minimum} $\lambda_i(L)$ is the smallest $\lambda$ such that $L$ contains $i$ linearly independent vectors of norm at most $\lambda$.
For $t\in\R^n$, let $\operatorname{dist}(t,L)=\min\{\|v-t\|: v\in L\}$.
The \emph{dual lattice} is $L^*=\{w\in\R^n: w\T y\in\Z \text{ for all } y\in L\}$.
We have $L^*=B^{-\mathsf T}\Z^n$ and $L^{**}=L$.

\paragraph{Promise problems.}
An algorithm \emph{solves} a promise problem if it runs within its time bound on every input and answers correctly on every input that satisfies the promise \cite[Section~1.2]{Goldreich06}.
Its output on other inputs is arbitrary.
A randomized or quantum algorithm solves a promise problem if it runs within its time bound on every input and answers correctly with probability at least $2/3$ on every input that satisfies the promise.
Similarly, an oracle for a promise problem answers every query, and answers correctly the queries that satisfy the promise.
A \emph{reduction} from one promise problem to another is a deterministic or randomized algorithm that solves the first using an oracle for the second, whatever the oracle answers on queries that violate the promise.
Each oracle call counts as one step, and the reduction is polynomial-time if this algorithm is.
It \emph{preserves the dimension} if, on inputs of dimension $n$, all its queries have dimension $n$.

\paragraph{Lattice problems.}
In the following problems, the input is a rational basis $B$ of a lattice $L\subseteq\R^n$, together with a rational vector $t$ for BDD and a rational number $\eta$ for GapSVP.
The parameters $\alpha\in(0,1)$ and $\gamma\ge1$ are functions of $n$.
\begin{itemize}
\item BDD$_\alpha$: under the promise $\operatorname{dist}(t,L)\le\alpha\lambda_1(L)$, find $v\in L$ with $\|v-t\|=\operatorname{dist}(t,L)$.
\item uSVP$_\gamma$: under the promise $\lambda_2(L)>\gamma\lambda_1(L)$, find $v\in L$ with $\|v\|=\lambda_1(L)$.
\item GapSVP$_\gamma$: under the promise that $\lambda_1(L)\le\eta$ or $\lambda_1(L)>\gamma\eta$, decide which of the two holds.
\item SIVP$_\gamma$: find $n$ linearly independent vectors of $L$ of norm at most $\gamma\lambda_n(L)$.
\end{itemize}

\paragraph{Learning with errors.}
Let $q\ge2$, and let $\Z_q=\Z/q\Z$, whose elements we represent by integers in $\{0,\dots,q-1\}$.
In the following problem, the input is a pair $(A,t)$ with $A\in\Z_q^{n\times m}$ and $t\in\Z_q^m$.
The modulus $q$, the number of samples $m$, and a probability distribution $\chi$ on $\Z$ are functions of $n$.
An \emph{LWE instance} and a \emph{uniform instance} are both pairs $(A,t)$ with $A$ uniformly random.
In an LWE instance, $t\equiv A\T s+e\pmod q$, where $s\in\Z_q^n$ is uniformly random and $e\in\Z^m$ has independent entries distributed according to $\chi$.
In a uniform instance, $t$ is uniformly random and independent of $A$.
\begin{itemize}
\item decision-$\LWE_{n,q,\chi}$ with $m$ samples: given an LWE instance or a uniform instance, decide which of the two it is, by outputting $1$ or $0$, respectively.
\end{itemize}
An algorithm \emph{solves} this problem if the probabilities that it outputs $1$ on an LWE instance and on a uniform instance differ by at least $n^{-c}$ for some constant $c>0$ and all sufficiently large $n$.
In our main results, $\chi$ is the error distribution $\bar\Psi_\alpha$ of Regev \cite[Section~2]{Regev09}, lifted from $\Z_q$ to $\Z$: for $\alpha\in(0,1)$, it is the distribution of $\lfloor X\rceil$, where $X$ is normal with mean zero and standard deviation $\alpha q/\sqrt{2\pi}$.
It can be sampled in polynomial time.

\paragraph{Chv\'atal rank exactly one.}
The questions of Cornu\'ejols, Lee, and Li are stated for polyhedra of Chv\'atal rank exactly one, while Questions~\ref{q:cube} and~\ref{q:cll} also allow rank zero.
The following standard lemma shows that this makes no difference.

\begin{lemma}\label{lem:exactlyone}
There is a polynomial-time algorithm that, given a system $Cx\le d$ with $C$ and $d$ rational, defining a polyhedron $P$, either decides whether $P$ contains an integral vector, or concludes that $P$ is not integral.
In particular, if $P'=P_I$ and the algorithm concludes that $P$ is not integral, then $P$ has Chv\'atal rank exactly one.
\end{lemma}
\begin{proof}
A rational polyhedron is integral if and only if each of its minimal faces contains an integral vector \cite[Section~16.3]{Schrijver86}.
We decide whether $P=\emptyset$ by linear programming \cite[Theorem~13.4]{Schrijver86}.
If $P=\emptyset$, then $P$ contains no integral vector.
Otherwise, we compute a system of linear equations $C'x=d'$ that defines a minimal face of $P$ \cite[proof of Theorem~16.2]{Schrijver86}, and we decide whether this face contains an integral vector \cite[Corollary~5.3b]{Schrijver86}.
If it does, then so does $P$.
Otherwise this minimal face contains no integral vector, so $P$ is not integral.
If moreover $P'=P_I$, then $P\ne P_I=P'$, so $P$ has Chv\'atal rank exactly one.
\end{proof}

Hence the versions of Questions~\ref{q:cube} and~\ref{q:cll} for polyhedra of Chv\'atal rank exactly one are equivalent to ours.
Indeed, an algorithm for our version applies in particular to polyhedra of Chv\'atal rank exactly one.
Conversely, given an algorithm for the version with Chv\'atal rank exactly one, we first process the input with Lemma~\ref{lem:exactlyone}, and we call this algorithm only if Lemma~\ref{lem:exactlyone} concludes that the input is not integral.

\section{Parallelepipeds of Chv\'atal rank at most one}\label{sec:slab}

In this section we prove Lemma~\ref{lem:slab}, the sufficient condition for Chv\'atal rank at most one on which the proofs of Theorems~A and~B rest, and we apply it to parallelepipeds defined by lattices (Proposition~\ref{cor:promise}).

\subsection{A sufficient condition}\label{sec:condition}

\begin{lemma}\label{lem:slab}
Let $U\in\Z^{n\times n}$ be unimodular with rows $u_1\T,\dots,u_n\T$, and let $P\subseteq\R^n$ be a rational polyhedron with $\width_P(u_i)<1$ for $i=1,\dots,n$.
Then $P'=P_I$, $P_I$ is the intersection of $P$ with at most $2n$ CG cuts, and $P$ contains at most one integral vector.
Moreover, every rational subpolyhedron of $P$ has Chv\'atal rank at most one.
\end{lemma}
\begin{proof}
We may assume $P\neq\emptyset$.
Then $\delta_P(\pm u_i)$ is finite, since $\width_P(u_i)<1$.
Let $P_0$ be the intersection of $P$ with the $2n$ CG cuts $u_i\T x\le\fl{\delta_P(u_i)}$ and $u_i\T x\ge-\fl{\delta_P(-u_i)}$.
The difference of these bounds is $\fl{\delta_P(u_i)}+\fl{\delta_P(-u_i)}$, an integer at most $\delta_P(u_i)+\delta_P(-u_i)=\width_P(u_i)<1$, and hence at most zero.
If $P_0\ne\emptyset$, it follows that, for every $i$, there is an integer $\gamma_i$ with $u_i\T x=\gamma_i$ for all $x\in P_0$, so $P_0=\{U^{-1}\gamma\}$.
Then $U^{-1}\gamma$ is an integral vector of $P$, since $U$ is unimodular.
Thus $P_0\subseteq P_I$, which also holds if $P_0=\emptyset$.
Since also $P_I\subseteq P'\subseteq P_0$, these three sets are equal.
In particular, $P$ contains at most one integral vector, since $P\cap\Z^n\subseteq P_I=P_0$.
Let $R$ be a rational subpolyhedron of $P$.
Then $\width_R\le\width_P$, so the width assumption holds for $R$, and the first statement applies to $R$.
\end{proof}

\begin{remark}\label{rem:croppedcube}
The strict inequality in Lemma~\ref{lem:slab} cannot be weakened to $\width_P(u_i)\le1$, and the last statement of Lemma~\ref{lem:slab} is not implied by $P'=P_I$.
Both facts are shown by the \emph{cropped cube}
\[
\Big\{x\in[0,1]^n:\sum_{j\in J}x_j+\sum_{j\notin J}(1-x_j)\ge\tfrac12\text{ for all }J\subseteq\{1,\dots,n\}\Big\},
\]
which has Chv\'atal rank at least $n$ \cite[Lemma~7.2]{ChvatalCookHartmann89}, and so at least two for $n\ge2$.
For such $n$, it has width one along every $e_i$, which shows the first fact, and it is a rational subpolyhedron of the integral cube $[0,1]^n$, which shows the second.
\end{remark}

\subsection{Parallelepipeds defined by lattices}\label{sec:parallelepipeds}

We now apply Lemma~\ref{lem:slab} to parallelepipeds defined by lattices: we compute their widths along the directions given by dual bases, and we show that these parallelepipeds satisfy the promise when they are thin (Proposition~\ref{cor:promise}).
For a basis $B\in\Q^{n\times n}$ of a lattice $L$, $t\in\Q^n$, and $r>0$, let
\[
Q(B,t,r)=\{x\in\R^n: \|Bx-t\|_\infty\le r\}=\{x\in\R^n: t-r\1\le Bx\le t+r\1\},
\]
which is a parallelepiped (Figure~\ref{fig:parallelepiped}).
The map $x\mapsto Bx$ is a bijection between $Q(B,t,r)\cap\Z^n$ and the set of vectors of $L$ within $\ell_\infty$-distance $r$ of $t$.
To apply Lemma~\ref{lem:slab} to $Q(B,t,r)$, we need a unimodular matrix along whose rows $Q(B,t,r)$ has width less than one.
The next lemma shows that every basis $w_1,\dots,w_n$ of the dual lattice $L^*$ gives a candidate, with rows $w_1\T B,\dots,w_n\T B$, and computes the widths along these rows.

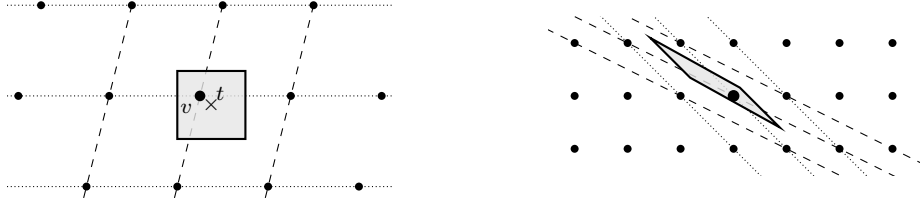
\begin{figure}[t]
\centering
\begin{tikzpicture}
\begin{scope}[x=0.3cm,y=0.3cm]
\begin{scope}\clip (-4.5,-4.5) rectangle (12.5,4.5);
\draw[dashed] (-1.125,-4.5) -- (1.125,4.5);
\draw[dashed] (2.875,-4.5) -- (5.125,4.5);
\draw[dashed] (6.875,-4.5) -- (9.125,4.5);
\draw[densely dotted] (-4.5,-4) -- (12.5,-4);
\draw[densely dotted] (-4.5,0) -- (12.5,0);
\draw[densely dotted] (-4.5,4) -- (12.5,4);
\end{scope}
\fill (-4,0) circle (1.5pt);
\fill (-3,4) circle (1.5pt);
\fill (-1,-4) circle (1.5pt);
\fill (0,0) circle (1.5pt);
\fill (1,4) circle (1.5pt);
\fill (3,-4) circle (1.5pt);
\fill (4,0) circle (1.5pt);
\fill (5,4) circle (1.5pt);
\fill (7,-4) circle (1.5pt);
\fill (8,0) circle (1.5pt);
\fill (9,4) circle (1.5pt);
\fill (11,-4) circle (1.5pt);
\fill (12,0) circle (1.5pt);
\draw[thick,fill=gray!20,fill opacity=0.8] (3,-1.9) rectangle (6,1.1);
\fill (4,0) circle (2pt) node[below left=-1pt] {\scriptsize$v$};
\draw (4.5,-0.4) node {\scriptsize$\times$} node[above right=-2pt] {\scriptsize$t$};
\end{scope}
\begin{scope}[shift={(8.95,-0.7)},x=0.7cm,y=0.7cm]
\begin{scope}\clip (-4.5,-0.5) rectangle (2.5,2.5);
\draw[dashed] (-7.5,3.75) -- (5.5,-2.75);
\draw[dashed] (-7.5,4.25) -- (5.5,-2.25);
\draw[dashed] (-7.5,4.75) -- (5.5,-1.75);
\draw[densely dotted] (-7.5,6.5) -- (5.5,-6.5);
\draw[densely dotted] (-7.5,7.5) -- (5.5,-5.5);
\draw[densely dotted] (-7.5,8.5) -- (5.5,-4.5);
\end{scope}
\draw[thick,fill=gray!20,fill opacity=0.8] (-1.819,1.344) -- (-2.569,2.094) -- (-0.881,1.156) -- (-0.131,0.406) -- cycle;
\fill (-4,0) circle (1.5pt);
\fill (-4,1) circle (1.5pt);
\fill (-4,2) circle (1.5pt);
\fill (-3,0) circle (1.5pt);
\fill (-3,1) circle (1.5pt);
\fill (-3,2) circle (1.5pt);
\fill (-2,0) circle (1.5pt);
\fill (-2,1) circle (1.5pt);
\fill (-2,2) circle (1.5pt);
\fill (-1,0) circle (1.5pt);
\fill (-1,1) circle (1.5pt);
\fill (-1,2) circle (1.5pt);
\fill (0,0) circle (1.5pt);
\fill (0,1) circle (1.5pt);
\fill (0,2) circle (1.5pt);
\fill (1,0) circle (1.5pt);
\fill (1,1) circle (1.5pt);
\fill (1,2) circle (1.5pt);
\fill (2,0) circle (1.5pt);
\fill (2,1) circle (1.5pt);
\fill (2,2) circle (1.5pt);
\fill (-1,1) circle (2.2pt);
\end{scope}
\end{tikzpicture}
\caption{The parallelepiped $Q(B,t,r)$ for $B=\left(\begin{smallmatrix}5&9\\4&4\end{smallmatrix}\right)$, $t=(4.5,-0.4)$, and $r=1.5$.
The vectors $w_1=(\frac14,-\frac1{16})$ and $w_2=(0,\frac14)$ form a basis of $L^*$, where $L=B\Z^2$, and $u_1=B\T w_1=(1,2)$ and $u_2=B\T w_2=(1,1)$.
Left: the lattice $L$, the square $t+r[-1,1]^2$, which contains only the vector $v=(4,0)$ of $L$, and the lines $w_1\T y=0,1,2$ (dashed) and $w_2\T y=-1,0,1$ (dotted).
Right: $Q(B,t,r)$, whose only integral vector is $B^{-1}v=(-1,1)$, and the lines $u_1\T x=0,1,2$ (dashed) and $u_2\T x=-1,0,1$ (dotted), which the map $x\mapsto Bx$ sends to the lines on the left.
By Lemma~\ref{lem:directions}, the widths of $Q(B,t,r)$ along $u_1$ and $u_2$ are $\frac{15}{16}$ and $\frac34$, so Lemma~\ref{lem:slab} applies, whereas its widths along $e_1$ and $e_2$ are $\frac{39}{16}$ and $\frac{27}{16}$.}
\label{fig:parallelepiped}
\end{figure}

\begin{lemma}\label{lem:directions}
Let $B\in\Q^{n\times n}$ be a basis of a lattice $L$, let $t\in\Q^n$, let $r>0$, and let $w_1,\dots,w_n$ be a basis of $L^*$.
Then $w_1\T B,\dots,w_n\T B$ are the rows of a unimodular matrix, every unimodular matrix arises in this way from a basis of $L^*$, and $\width_{Q(B,t,r)}(B\T w)=2r\|w\|_1$ for every $w\in\R^n$.
\end{lemma}
\begin{proof}
The map $w\mapsto B\T w$ is a linear bijection from $L^*$ to $\Z^n$, hence it maps the bases of $L^*$ bijectively to the bases of $\Z^n$.
The transposes of the vectors of a basis of $\Z^n$ are the rows of a unimodular matrix, and every unimodular matrix arises in this way.
Moreover, $(B\T w)\T x=w\T Bx$ and $\{Bx: x\in Q(B,t,r)\}=t+r[-1,1]^n$.
Hence $\delta_{Q(B,t,r)}(B\T w)=w\T t+r\|w\|_1$ for every $w\in\R^n$, and adding this identity for $w$ and $-w$ gives the claim.
\end{proof}

To make these widths less than one, we need a basis of $L^*$ with $\ell_1$-norms less than $1/(2r)$.
The next lemma gives a basis with $\ell_1$-norms at most $n^2/\lambda_1(L)$, which suffices when $r<\lambda_1(L)/(2n^2)$.

\begin{lemma}\label{lem:dualbasis}
For every lattice $L\subseteq\R^n$, the dual lattice $L^*$ has a basis $w_1,\dots,w_n$ with $\|w_i\|_1\le n^2/\lambda_1(L)$ for $i=1,\dots,n$.
\end{lemma}
\begin{proof}
We have $\lambda_n(L^*)\le n/\lambda_1(L)$, by a transference theorem of Banaszczyk \cite[Theorem~2.1]{Banaszczyk93}.
Moreover, by \cite[Corollary~7.2]{MicciancioGoldwasser02}, $L^*$ has a basis $w_1,\dots,w_n$ such that, for $i=1,\dots,n$,
\[
\|w_i\|\le\max\{1,\sqrt i/2\}\,\lambda_i(L^*).
\]
Hence $\|w_i\|\le\max\{1,\sqrt i/2\}\,\lambda_i(L^*)\le\sqrt n\,\lambda_n(L^*)\le n^{3/2}/\lambda_1(L)$, and $\|w_i\|_1\le\sqrt n\,\|w_i\|\le n^2/\lambda_1(L)$.
\end{proof}

We call $Q(B,t,r)$ \emph{thin} if $r\in\Q$ and $0<r<\lambda_1(L)/(2n^2)$, where $L\subseteq\R^n$ is the lattice with basis $B$.
Combining Lemmas~\ref{lem:slab}, \ref{lem:directions}, and~\ref{lem:dualbasis}, we obtain the following.

\begin{proposition}\label{cor:promise}
If $Q(B,t,r)$ is thin, then every rational subpolyhedron of $Q(B,t,r)$ has Chv\'atal rank at most one, and $Q(B,t,r)$ contains at most one integral vector.
\end{proposition}
\begin{proof}
Let $w_1,\dots,w_n$ be the basis of $L^*$ given by Lemma~\ref{lem:dualbasis}.
By Lemma~\ref{lem:directions}, the matrix with rows $w_1\T B,\dots,w_n\T B$ is unimodular, and, for each $i$, the width of $Q(B,t,r)$ along $B\T w_i$ is $2r\|w_i\|_1\le2rn^2/\lambda_1(L)<1$.
Since $B$, $t$, and $r$ are rational, $Q(B,t,r)$ is a rational polyhedron, and the statement follows from Lemma~\ref{lem:slab}.
\end{proof}

\section{Hardness for parallelepipeds}\label{sec:parhard}

In this section we reduce worst-case lattice problems and learning with errors to integer feasibility under the promise $P'=P_I$, using the parallelepipeds of Proposition~\ref{cor:promise}.
This proves Theorems~A and~B, through Theorems~\ref{thm:bdd} and~\ref{cor:lwe}, and gives their consequences for lattice problems (Corollaries~\ref{cor:usvp} and~\ref{cor:regev}).

\subsection{Hardness from worst-case lattice problems}\label{sec:hardness}

To solve BDD, we try radii $r$ that are powers of two, and we find the closest vector coordinate by coordinate (Theorem~\ref{thm:bdd}).
Combined with reductions of Lyubashevsky and Micciancio, this gives consequences for uSVP and GapSVP (Corollary~\ref{cor:usvp}).

\begin{theorem}\label{thm:bdd}
BDD$_{1/(4n^2)}$ can be solved in deterministic polynomial time with an oracle for integer feasibility under the promise that the query is a thin parallelepiped in $\R^n$.
Every query is a rational parallelepiped in $\R^n$.
\end{theorem}
\begin{proof}
Let $B\in\Q^{n\times n}$ be a basis of a lattice $L$ and let $t\in\Q^n$, with total encoding size $\sigma$.
Let $\rho=\operatorname{dist}(t,L)$, and let $v^*$ be a vector of $L$ closest to $t$.
For an integer $j$, let $Q_j=Q(B,t,2^j)$.
The idea is to guess $\rho$ up to a factor of two by trying the radii $2^j$.
If the input satisfies the promise of BDD, some $j$ satisfies $\rho\le2^j<\lambda_1(L)/(2n^2)$.
For such $j$, $Q_j$ is thin, and $B^{-1}v^*$ is its only integral vector, and we find its entries bit by bit, with queries for thin parallelepipeds defined by sublattices of $L$.

We first construct integers $j_0\le j_1$ such that, under the promise, some $j^*$ between them satisfies $\rho\le2^{j^*}<\lambda_1(L)/(2n^2)$.
To do so, it suffices that $2^{j_0}\le\lambda_1(L)/(4n^2)$ and $2^{j_1}\ge\rho$.
Indeed, let $j^*$ be the smallest integer $j\ge j_0$ with $2^j\ge\rho$, so that $j^*\le j_1$.
If $j^*=j_0$, then $2^{j^*}\le\lambda_1(L)/(4n^2)$.
Otherwise $2^{j^*-1}<\rho$ by the choice of $j^*$, and so $2^{j^*}<2\rho$, which is at most $\lambda_1(L)/(2n^2)$ under the promise $\rho\le\lambda_1(L)/(4n^2)$.
We now construct such $j_0$ and $j_1$.
Since $\0\in L$, we have $\rho\le\|t\|\le2^\sigma$, so $j_1=\sigma$ satisfies $2^{j_1}\ge\rho$.
The least common denominator of the entries of $B$ is at most $2^\sigma$, and multiplying $L$ by it gives a sublattice of $\Z^n$, whose nonzero vectors have norm at least one.
Thus $\lambda_1(L)\ge2^{-\sigma}$.
Hence $j_0=-\sigma-\lceil\log_2(4n^2)\rceil$ satisfies $2^{j_0}\le\lambda_1(L)/(4n^2)$.

Next we describe the queries and bound the running time.
We consider $j=j_0,j_0+1,\dots,j_1$ and do the following for each $j$.
We ask the oracle whether $Q_j$ contains an integral vector, and if it reports that $Q_j$ contains none, we proceed to the next $j$.
If the oracle reports that $Q_j$ contains an integral vector, then, for $i=1,\dots,n$, we compute $I_i=\{x_i: x\in Q_j\}$ by linear programming, we let $c_i=\lceil\min I_i\rceil$, and we let $K_i$ be the smallest nonnegative integer with $2^{K_i}>\max I_i-c_i$.
For $k=1,\dots,K_i$, let $B^{(i,k)}$ be the matrix obtained from $B$ by multiplying its $i$th column $b_i$ by $2^k$.
Starting from $u_{i,0}=0$, for $k=1,\dots,K_i$ in turn, we ask the oracle whether $Q(B^{(i,k)},t-(c_i+u_{i,k-1})b_i,2^j)$ contains an integral vector, and we let $u_{i,k}=u_{i,k-1}$ if it reports that it does and $u_{i,k}=u_{i,k-1}+2^{k-1}$ otherwise.
We let $z_i=c_i+u_{i,K_i}$, output $Bz$, and stop.
If no output is produced for $j\le j_1$, we output $\0$.
Each query is a parallelepiped $Q(B',t',2^j)$ in $\R^n$, given by the $2n$ inequalities $t'-2^j\1\le B'x\le t'+2^j\1$, where $B'$ is $B$ or some $B^{(i,k)}$, and so is rational, since $B'$, $t'$, and $2^j$ are.
There are $j_1-j_0+1=O(\sigma)$ values of $j$, and $2^j$ has encoding size $O(\sigma)$ for each of them.
The endpoints of $I_i$ are optimal values of linear programs over $Q_j$, so they have encoding size polynomial in $\sigma$ \cite[Corollary~10.2a]{Schrijver86}, so $K_i$ is polynomial in $\sigma$, and $B^{(i,k)}$ and $t-(c_i+u_{i,k-1})b_i$ have encoding size polynomial in $\sigma$.
Hence, on every input, the oracle is called a number of times polynomial in $\sigma$, on queries of encoding size polynomial in $\sigma$, and all other computations take time polynomial in $\sigma$.

We now show that the output is $v^*$ if the input satisfies the promise $\rho\le\lambda_1(L)/(4n^2)$.
The vector $z^*=B^{-1}v^*$ is integral, since $v^*\in B\Z^n$, and it lies in $Q_{j^*}$, since $\|Bz^*-t\|_\infty=\|v^*-t\|_\infty\le\|v^*-t\|=\rho\le2^{j^*}$.
Since $2^{j^*}<\lambda_1(L)/(2n^2)$, the parallelepiped $Q_{j^*}$ is thin, so $z^*$ is its unique integral vector by Proposition~\ref{cor:promise}.
For $j\le j^*$, the parallelepiped $Q_j$ is thin, since $2^j\le2^{j^*}$, so the oracle answers the query $Q_j$ correctly, and $Q_j\cap\Z^n\subseteq Q_{j^*}\cap\Z^n=\{z^*\}$.
Hence, if $Q_j$ contains no integral vector, we proceed to the next $j$.
Otherwise $Q_j\cap\Z^n=\{z^*\}$, that is, $Bz^*$ is the only vector of $L$ within $\ell_\infty$-distance $2^j$ of $t$.
Fix $i$ and $k$, and let $u=u_{i,k-1}$.
Since $B^{(i,k)}x+(c_i+u)b_i=By$ for $y_i=2^kx_i+c_i+u$ and $y_l=x_l$ for $l\ne i$, the integral vectors $x$ of $Q(B^{(i,k)},t-(c_i+u)b_i,2^j)$ correspond to the integral vectors $y$ with $\|By-t\|_\infty\le2^j$ and $y_i\equiv c_i+u\pmod{2^k}$.
So this parallelepiped contains an integral vector if and only if $z^*_i\equiv c_i+u\pmod{2^k}$.
It is thin, since $B^{(i,k)}$ is a basis of a sublattice of $L$, whose nonzero vectors have norm at least $\lambda_1(L)>2n^2\cdot2^j$, so the oracle answers it correctly.
Since $z^*_i$ is an integer in $I_i$, we have $0\le z^*_i-c_i<2^{K_i}$, and induction on $k$ gives $u_{i,k}=(z^*_i-c_i)\bmod2^k$, so $z_i=z^*_i$.
Hence $z=z^*$.
Since $Q_{j^*}$ contains $z^*$, the procedure stops at some $j\le j^*$ and outputs $Bz^*=v^*$.
This is why we try the values of $j$ in increasing order: for $j>j^*$, the queries may violate the promise, and the oracle may answer them incorrectly.
\end{proof}

\begin{proof}[Proof of Theorem~A]
By Proposition~\ref{cor:promise}, every thin parallelepiped in $\R^n$ satisfies $P'=P_I$, and it is rational by definition.
By Theorem~\ref{thm:bdd}, every query is a rational parallelepiped in $\R^n$.
Hence Theorem~\ref{thm:bdd} gives Theorem~A.
\end{proof}

We use the following reductions of Lyubashevsky and Micciancio.

\begin{theorem}[{\cite[Theorems~2 and~4 and their proofs]{LyubashevskyMicciancio09}}]\label{thm:lm}
Let $\gamma=\gamma(n)$ be bounded by a polynomial in $n$.
There is a deterministic polynomial-time reduction from uSVP$_\gamma$ to BDD$_{1/\gamma}$ that preserves the dimension.
If $n\ge2$ and $\gamma>2\sqrt{n/\log n}$, there is a randomized polynomial-time reduction from GapSVP$_\gamma$ to BDD$_{\sqrt{n/\log n}/\gamma}$ that preserves the dimension.
\end{theorem}

The proof of \cite[Theorem~2]{LyubashevskyMicciancio09} assumes that the oracle returns a lattice vector on every query.
Every oracle for BDD can be made to do so by replacing each answer $v$ to a query with basis $B$ such that $B^{-1}v\notin\Z^n$ by $\0$.
This does not change the answers to queries that satisfy the promise, so both results give reductions in the sense of Section~\ref{sec:prelim}.
Replacing the oracle by an algorithm and combining Theorem~\ref{thm:bdd} with Theorem~\ref{thm:lm}, we obtain the further consequences stated in the introduction.

\begin{corollary}\label{cor:usvp}
Suppose that an algorithm solves integer feasibility for rational parallelepipeds under the promise that every rational subpolyhedron of the input, and in particular the input itself, has Chv\'atal rank at most one, in time $T(n,\sigma')$, where $n$ is the dimension, $\sigma'$ is the encoding size of the input, and $T$ is nondecreasing.
If $T$ is bounded by a polynomial, then BDD$_{1/(4n^2)}$ and uSVP$_{4n^2}$ can be solved in deterministic polynomial time, and, for $n\ge2$, GapSVP$_\gamma$ with $\gamma=4n^{5/2}/\sqrt{\log n}$ can be solved in randomized polynomial time.
If $T(n,\sigma')=2^{o(n)}\poly(\sigma')$, then these problems can be solved in time $2^{o(n)}$ times a polynomial in the encoding size, deterministically for BDD and uSVP and with randomization for GapSVP.
If the algorithm is randomized, then all these statements hold with randomized algorithms for the three problems.
\end{corollary}
\begin{proof}
By Proposition~\ref{cor:promise}, every thin parallelepiped satisfies this promise, so, as in the proof of Theorem~A, the algorithm can serve as the oracle in Theorem~\ref{thm:bdd}.
Since the queries are in dimension $n$ and have encoding size polynomial in the encoding size $\sigma$ of $(B,t)$, and $T$ is nondecreasing, we obtain a deterministic algorithm for BDD$_{1/(4n^2)}$ that runs in time $T(n,\poly(\sigma))\poly(\sigma)$ on every input.
The statements about uSVP and GapSVP follow from Theorem~\ref{thm:lm} with $\gamma=4n^2$ and $\gamma=4n^{5/2}/\sqrt{\log n}$, respectively, since in both cases the BDD factor is $1/(4n^2)$.
The statements about running times $2^{o(n)}$ hold since all the reductions preserve the dimension and make a polynomial number of queries of polynomial encoding size.

Suppose now that the algorithm is randomized and answers correctly with probability at least $2/3$.
The algorithms above make a number $K$ of queries polynomial in the encoding size of their input.
We answer each new query by the majority answer of $O(\log K)$ independent runs of the algorithm, and give the same answer if the query is repeated, so that, by Hoeffding's inequality, each query that satisfies the promise is answered correctly with probability at least $1-1/(12K)$.
By the union bound, all such queries are answered correctly with probability at least $11/12$.
Hence BDD and uSVP are solved correctly with probability at least $11/12$, and GapSVP with probability at least $2/3-1/12=7/12$.
Taking the majority answer of a constant number of independent runs of this algorithm for GapSVP raises this probability to $2/3$, again by Hoeffding's inequality.
\end{proof}

\subsection{Hardness from learning with errors}\label{sec:lwe}

To solve decision-LWE, we use a single parallelepiped defined by a random \emph{$q$-ary lattice}, that is, a lattice $L$ with $q\Z^m\subseteq L\subseteq\Z^m$ \cite[Section~2]{MicciancioRegev09}.
With high probability, it is thin, and so satisfies the promise, and it contains an integral vector for LWE instances but not for uniform instances (Proposition~\ref{thm:lwe}).
Thus a single call to an algorithm for integer feasibility distinguishes the two cases (Theorem~\ref{cor:lwe}).
Combined with a result of Regev, this gives consequences for GapSVP and SIVP for quantum algorithms (Corollary~\ref{cor:regev}).

For $A\in\Z_q^{n\times m}$ we consider the $q$-ary lattice
\[
\Lambda_q(A\T)=\{y\in\Z^m: y\equiv A\T s \pmod q \text{ for some } s\in\Z_q^n\},
\]
which contains $q\Z^m$ and hence has full rank.
For an LWE instance $(A,t)$ with error $e$, we have $t-e\in\Lambda_q(A\T)$, so $t$ is close to $\Lambda_q(A\T)$ when $e$ is short.
In this subsection, $n$ is the dimension of $s$, and also of the lattice problems in Theorem~\ref{thm:regev}, while the lattices $\Lambda_q(A\T)$ and the parallelepipeds $Q(B,t,r)$ have dimension $m$.

By Proposition~\ref{cor:promise}, the parallelepipeds defined by $\Lambda_q(A\T)$ satisfy the promise whenever they are thin, and they are thin if their radius is small compared with $\lambda_1(\Lambda_q(A\T))$.
We first show, by a standard counting argument, that $\lambda_1(\Lambda_q(A\T))$ is large with high probability over the choice of $A$.

\begin{lemma}\label{lem:lambda1}
Let $q$ be prime, and let $A\in\Z_q^{n\times m}$ be uniformly random.
Then $\lambda_1(\Lambda_q(A\T))>q/8$ with probability at least $1-q^n2^{-m}$.
\end{lemma}
\begin{proof}
Let $L=\Lambda_q(A\T)$.
A nonzero vector of $q\Z^m$ has $\ell_\infty$-norm at least $q$, and every other vector of $L$ is congruent modulo $q$ to $A\T s$ for some nonzero $s\in\Z_q^n$.
For fixed nonzero $s$, the vector $A\T s$ is uniformly random in $\Z_q^m$, since $q$ is prime.
At most $2\fl{q/8}+1\le q/2$ residues modulo $q$ have a representative $c$ with $|c|\le q/8$.
Hence, for fixed nonzero $s$, all entries of $A\T s$ are such residues with probability at most $2^{-m}$.
By the union bound over the fewer than $q^n$ nonzero $s$, the probability that $L$ has a nonzero vector of $\ell_\infty$-norm at most $q/8$ is at most $q^n2^{-m}$.
Otherwise every nonzero $y\in L$ satisfies $\|y\|\ge\|y\|_\infty>q/8$, and so $\lambda_1(L)>q/8$.
\end{proof}

We now show that the parallelepiped defined by an instance $(A,t)$ is thin with high probability, and that it distinguishes LWE instances from uniform instances.

\begin{proposition}\label{thm:lwe}
Let $q$ be prime, let $k=\lceil\log_2q\rceil$ and $m=3nk$, and assume $q>16m^2$.
Let $r=\fl{(q-1)/(16m^2)}$, let $\chi$ be a probability distribution on $\Z$, and let $\varepsilon_\chi$ be the probability that a vector $e\in\Z^m$ with independent entries distributed according to $\chi$ satisfies $\|e\|_\infty>r$.
Given $A$, a basis $B\in\{0,\dots,q\}^{m\times m}$ of $\Lambda_q(A\T)$ can be computed in time polynomial in $n$ and $\log q$.
For $t\in\Z_q^m$, let $Q=Q(B,t,r)$.
If $(A,t)$ is an LWE instance or a uniform instance, then $Q$ is thin with probability at least $1-2^{-2nk}$, and $Q$ contains an integral vector with probability at least $1-\varepsilon_\chi$ in the first case and at most $2^{-m}$ in the second.
\end{proposition}
\begin{proof}
We first compute $B$ and $Q$.
We let $B$ be the Hermite normal form basis of the lattice generated by the columns of $A\T$ and the vectors $qe_1,\dots,qe_m$, that is, of $\Lambda_q(A\T)$, which can be computed in time $\poly(n,\log q)$ \cite[Theorem~5.3]{Schrijver86}.
For every $i$, we have $qe_i\in\Lambda_q(A\T)$, so $qe_i=By$ for some $y\in\Z^m$.
Since $B$ is lower triangular, the first $i$ rows of this system give $y_1=\dots=y_{i-1}=0$ and $B_{ii}y_i=q$.
Hence each diagonal entry of $B$ divides $q$, and the other entries in its row are nonnegative and smaller, so $B\in\{0,\dots,q\}^{m\times m}$.
Then $Q=\{x: Bx\le t+r\1,\ -Bx\le -t+r\1\}$ has integral data, $2m$ inequalities, and encoding size polynomial in $n$ and $\log q$.

Next we show that $Q$ is thin with high probability.
In LWE instances and in uniform instances alike, $A$ is uniformly random, so $\lambda_1(\Lambda_q(A\T))>q/8$ with probability at least $1-q^n2^{-m}$ by Lemma~\ref{lem:lambda1}.
Moreover, $q^n\le2^{nk}=2^{m/3}$, so this probability is at least $1-2^{-2nk}$.
Since $q>16m^2$, we have $r\ge1$ and $r<q/(16m^2)$.
Hence, if $\lambda_1(\Lambda_q(A\T))>q/8$, then $r<(q/8)/(2m^2)<\lambda_1(\Lambda_q(A\T))/(2m^2)$, and so $Q$ is thin.

We then show that, for an LWE instance, $Q$ contains an integral vector with probability at least $1-\varepsilon_\chi$.
By the definition of $\varepsilon_\chi$, the error $e$ satisfies $\|e\|_\infty\le r$ with probability at least $1-\varepsilon_\chi$.
In this case, since $t\equiv A\T s+e\pmod q$, the vector $t-e\in\Lambda_q(A\T)$ lies within $\ell_\infty$-distance $r$ of $t$, and so $Q$ contains an integral vector.

Finally, we show that, for a uniform instance, $Q$ contains an integral vector with probability at most $2^{-m}$.
If $Q$ contains an integral vector, then $t\equiv A\T s+e'\pmod q$ for some $s\in\Z_q^n$ and $e'\in\{-r,\dots,r\}^m$.
There are at most $q^n(2r+1)^m$ such residues, and so $Q$ contains an integral vector with probability at most $q^n(2r+1)^m/q^m$.
Since $q>16m^2\ge16$, we have $2r+1\le q/(8m^2)+1\le q/4$, and $q^n\le2^{nk}=2^{m/3}$, so this probability is at most $2^{m/3}4^{-m}\le2^{-m}$.
\end{proof}

For the error distribution $\bar\Psi_\alpha$ and suitable parameters, Proposition~\ref{thm:lwe} gives an explicit distribution of hard instances.

\begin{theorem}\label{cor:lwe}
Let $n$ be sufficiently large, let $q=q(n)$ be a prime in $[n^3,2n^3]$, which exists by Bertrand's postulate, let $k=\lceil\log_2q\rceil$, let $\alpha=1/(200n^2k^2\log n)$, and let $m=3nk$.
Let $\mathcal D$ be the distribution of the parallelepiped $Q$ of Proposition~\ref{thm:lwe} with $\chi=\bar\Psi_\alpha$ when $(A,t)$ is an LWE instance or a uniform instance, each with probability $1/2$.
The distribution $\mathcal D$ can be sampled in polynomial time, its samples are parallelepipeds in $\R^m$ given by integral constraints with coefficients of absolute value less than $2q$, and a sample of $\mathcal D$ is thin with probability at least $1-2^{-2nk}$.
Suppose that a randomized polynomial-time algorithm decides integer feasibility correctly with probability at least $1/2+n^{-c}$ for some constant $c>0$, over the choice of $Q$ from $\mathcal D$ and its own randomness.
Then there is a randomized polynomial-time algorithm that solves decision-$\LWE_{n,q,\bar\Psi_\alpha}$ with $m$ samples.
\end{theorem}
\begin{proof}
Since $n$ is sufficiently large, we have $q>16m^2$, so Proposition~\ref{thm:lwe} applies with $\chi=\bar\Psi_\alpha$.

We first prove the properties of $\mathcal D$ stated in the theorem.
The distribution $\mathcal D$ can be sampled in polynomial time, since $q$ can be found in time polynomial in $n$ by testing the integers in $[n^3,2n^3]$ for primality, $\bar\Psi_\alpha$ can be sampled in polynomial time, and $Q$ can be computed in polynomial time.
Its samples are the parallelepipeds $Q$ in $\R^m$, given by the integral constraints $t-r\1\le Bx\le t+r\1$, whose coefficients have absolute value less than $2q$, since $B\in\{0,\dots,q\}^{m\times m}$, $t\in\{0,\dots,q-1\}^m$, and $r<q$.
By Proposition~\ref{thm:lwe}, a sample of $\mathcal D$ is thin with probability at least $1-2^{-2nk}$.

Next we show that, for an LWE instance with error $e$, $Q$ contains an integral vector with probability at least $1-2m\exp(-\pi\log^2n)$.
By Proposition~\ref{thm:lwe}, this probability is at least $1-\varepsilon_\chi$.
Moreover, $r=\fl{(q-1)/(16m^2)}\ge\alpha q\log n$ for sufficiently large $n$, since $\alpha q\log n=q/(200n^2k^2)$, $16m^2=144n^2k^2$, $q\ge n^3$, and $k=O(\log n)$.
Each entry of $e$ is distributed according to $\bar\Psi_\alpha$, that is, it equals $\lfloor X\rceil$ with $X$ normal with mean zero and standard deviation $\bar\sigma=\alpha q/\sqrt{2\pi}$.
Then each entry has absolute value greater than $r$ with probability at most $\Pr[|X|\ge r+\frac12]\le\Pr[|X|\ge\alpha q\log n]=\Pr[|X|\ge\sqrt{2\pi}\log n\cdot\bar\sigma]\le2\exp(-\pi\log^2n)$, by the Gaussian tail bound $\Pr[|X|\ge\tau\bar\sigma]\le2\exp(-\tau^2/2)$ for $\tau\ge0$.
Hence, by the union bound over the $m$ entries, $\varepsilon_\chi\le2m\exp(-\pi\log^2n)$, which is smaller than every inverse polynomial in $n$ for sufficiently large $n$.

It remains to construct the algorithm for decision-LWE.
Given $(A,t)$, our algorithm computes $Q$ as in Proposition~\ref{thm:lwe}, runs the given algorithm on $Q$, and outputs $1$ if and only if the given algorithm reports that $Q$ contains an integral vector.
Since $q\le2n^3$, our algorithm runs in time polynomial in $n$.
Let $\beta_1$ and $\beta_0$ be the probabilities that the given algorithm answers incorrectly on $Q$ for LWE instances and for uniform instances, respectively.
Then $\beta_1+\beta_0\le1-2n^{-c}$, since the given algorithm answers incorrectly with probability at most $1/2-n^{-c}$ for $Q$ drawn from $\mathcal D$.
By Proposition~\ref{thm:lwe}, the output is $1$ with probability at least $1-\varepsilon_\chi-\beta_1$ on LWE instances and at most $2^{-m}+\beta_0$ on uniform instances.
These probabilities differ by at least $2n^{-c}-\varepsilon_\chi-2^{-m}$, which is at least $n^{-c}$ for sufficiently large $n$.
\end{proof}

\begin{proof}[Proof of Theorem~B]
By Proposition~\ref{cor:promise}, thin parallelepipeds satisfy $P'=P_I$.
So a sample of $\mathcal D$ satisfies $P'=P_I$ with high probability.
Moreover, $2/3\ge1/2+1/n$ for $n\ge6$.
Hence Theorem~\ref{cor:lwe} gives Theorem~B.
\end{proof}

We use the following result of Regev.

\begin{theorem}[{\cite[Theorem~3.1, Section~3.3, and Lemma~4.4]{Regev09}}]\label{thm:regev}
Let $q=q(n)$ be a prime and $m=m(n)$ an integer, both bounded by a polynomial in $n$, and let $\alpha=\alpha(n)\in(0,1)$ with $\alpha q>2\sqrt n$ for all sufficiently large $n$.
If there is a randomized polynomial-time algorithm that solves decision-$\LWE_{n,q,\bar\Psi_\alpha}$ with $m$ samples, then there is a polynomial-time quantum algorithm that solves GapSVP$_\gamma$ and SIVP$_\gamma$ on $n$-dimensional lattices for some $\gamma=\tilde O(n/\alpha)$.
\end{theorem}

Theorem~\ref{thm:regev} is stated with our notion of solving decision-LWE.
Regev's results apply to it since, by averaging over $s$, an algorithm that solves decision-LWE in our sense distinguishes, for a fraction at least $n^{-c}/2$ of all $s\in\Z_q^n$, the LWE instances with secret $s$ from uniform instances, with output probabilities that differ by at least $n^{-c}/2$, as required in \cite[Lemma~4.4]{Regev09}.
Combining Theorem~\ref{cor:lwe} with Theorem~\ref{thm:regev}, we obtain the consequence for quantum algorithms stated in the introduction.

\begin{corollary}\label{cor:regev}
Suppose that there is a randomized polynomial-time algorithm as in Theorem~\ref{cor:lwe}.
Then, for some $\gamma=\tilde O(n^3)$, there is a polynomial-time quantum algorithm that solves GapSVP$_\gamma$ and SIVP$_\gamma$ on $n$-dimensional lattices.
\end{corollary}
\begin{proof}
Let $q$, $k$, $\alpha$, and $m$ be as in Theorem~\ref{cor:lwe}.
By Theorem~\ref{cor:lwe}, there is a randomized polynomial-time algorithm that solves decision-$\LWE_{n,q,\bar\Psi_\alpha}$ with $m$ samples.
Theorem~\ref{thm:regev} applies, since $q\le2n^3$ is prime, $q$ and $m=3nk$ are bounded by a polynomial in $n$, and $\alpha q\ge n^3/(200n^2k^2\log n)>2\sqrt n$ for sufficiently large $n$, and it gives $\gamma=\tilde O(n/\alpha)=\tilde O(n^3)$.
\end{proof}

\subsection{Refutations}\label{sec:parremarks}

We conclude with a remark on \emph{refutations}, that is, sets of CG cuts for a polyhedron with empty intersection.

\begin{remark}\label{rem:refutation}
By Proposition~\ref{cor:promise}, thin parallelepipeds and their rational subpolyhedra that contain no integral vector have empty Chv\'atal closure.
Since the Chv\'atal closure is the intersection of finitely many CG cuts \cite{Schrijver80}, Helly's theorem then gives them refutations with at most $n+1$ CG cuts.
Given a short dual basis, refutations can be computed in polynomial time, but both refutations and short dual bases are hard to find.

We first show that a short dual basis gives a refutation.
Let $Q(B,t,r)$ be a thin parallelepiped, let $L$ be the lattice with basis $B$, let $P\subseteq Q(B,t,r)$ be a rational polyhedron without integral vectors, and let $w_1,\dots,w_n$ be the basis of $L^*$ given by Lemma~\ref{lem:dualbasis}.
If $P=\emptyset$, then the CG cut $\0\T x\le-1$ is a refutation, so we may assume $P\neq\emptyset$.
As in the proof of Proposition~\ref{cor:promise}, $P$ has width less than one along each $u_i=B\T w_i$, so the CG cuts for $P$ along $u_i$ and $-u_i$ either have empty intersection or together give an equation $u_i\T x=\gamma_i$ with $\gamma_i\in\Z$.
If they give an equation for every $i$, the equations define a single integral vector $z$, since $u_1\T,\dots,u_n\T$ are the rows of a unimodular matrix, and $z\notin P$.
An inequality defining $P$ that $z$ violates, scaled to have integral coefficients and with its right-hand side rounded down, is a CG cut violated by $z$, and together with the $n$ equations it gives a contradiction.
Hence $P$ has a refutation with at most $2n+1$ CG cuts.
Given $w_1,\dots,w_n$, these cuts can be computed in polynomial time.
For $P=Q(B,t,r)$, the cuts along $u_i$ are $u_i\T x\le\fl{w_i\T t+r\|w_i\|_1}$ and $u_i\T x\ge-\fl{-w_i\T t+r\|w_i\|_1}$, as in Babai's round-off algorithm \cite{Babai86}.

We now show that refutations and short dual bases are hard to find.
Given an algorithm that finds a refutation whenever one exists, we decide integer feasibility for thin parallelepipeds by running it and checking by linear programming whether its output is a refutation.
So, by Theorem~\ref{thm:bdd}, finding any refutation is as hard as BDD$_{1/(4n^2)}$.
The same holds for finding $w_1,\dots,w_n$, or any unimodular matrix along whose rows $Q(B,t,r)$ has width less than one, since such a matrix gives, by linear programming, the CG cuts of the proof of Lemma~\ref{lem:slab}, and thus decides whether $Q(B,t,r)$ contains an integral vector.
Similarly, since the parallelepipeds of Theorem~\ref{cor:lwe} are thin with high probability, finding any refutation, or such a matrix, for them is as hard as decision-LWE.
\end{remark}

\section{SAT polytopes of Chv\'atal rank at most one}\label{sec:sat}

In this section we construct SAT polytopes defined by lattices, on which the proofs of Theorems~C and~D rest, and we show that they satisfy the promise when the lattice has no short nonzero vectors (Proposition~\ref{prop:satlattice}).

A \emph{literal} on variables $x_1,\dots,x_N$ is an affine function $x_i$ or $\bar x_i=1-x_i$.
For a literal $\ell$, we write $\bar\ell=1-\ell$, which is again a literal.
A \emph{formula} $F$ in conjunctive normal form is a finite family of \emph{clauses}, each of which is a nonempty set of literals.
The \emph{SAT polytope} of $F$ is
\[
P(F)=\Big\{x\in[0,1]^N:\ \sum_{\ell\in T}\ell(x)\ge1\text{ for every }T\in F\Big\}.
\]
Its integral vectors are the \emph{satisfying assignments} of $F$.
Without the promise $P(F)'=P(F)_I$, satisfiability is NP-complete for formulas whose clauses have at most three literals \cite{Cook71}.

The proofs of Theorems~C and~D have the same structure as those of Theorems~A and~B.
The parallelepipeds $Q(B,t,r)$ are replaced by the formulas of Proposition~\ref{prop:satlattice} below, which also takes the place of Proposition~\ref{cor:promise}.
Unlike Proposition~\ref{cor:promise}, Proposition~\ref{prop:satlattice} does not follow directly from Lemma~\ref{lem:slab}, and we prove it in four parts.
In Section~\ref{sec:exactone}, we show that the Chv\'atal closure of our SAT polytopes lies in a face defined by \emph{exact-one equations}, which set a sum of literals equal to one.
In Section~\ref{sec:circuits}, we model arithmetic on binary numbers by such equations, through extended formulations, and we show that its identities hold on the whole face.
These identities give the integral linear functions along which we take CG cuts.
In Section~\ref{sec:residue}, we encode each entry of the error by its residues modulo small pairwise coprime integers, and binary arithmetic recovers the error and checks that the corresponding vector lies in the lattice.
In Section~\ref{sec:thinformulas}, we show that, for each residue, a vector of the dual lattice close to a unit vector divided by the modulus gives integral linear functions of width less than one.
As in the proof of Lemma~\ref{lem:slab}, CG cuts along them fix all residues, and hence leave at most one integral vector (Proposition~\ref{thm:crt}).

Residues are needed because the errors can be exponentially large: our formulas need an integral lattice, and making a general rational lattice integral can create such errors (Lemma~\ref{lem:integralize}).
Such errors cannot be written in unary with polynomially many $0/1$ variables.
In binary, the digits of the errors depend on residues modulo powers of two as large as the errors themselves, which our CG cuts cannot fix.
Residues modulo small integers represent the errors compactly, and one round of CG cuts can fix these residues.
The price is a factor equal to the largest modulus in the condition on $\lambda_1$ in Proposition~\ref{prop:satlattice}(b).

The lattices in our formulas are given by congruences.
For integers $D\ge1$, $m\ge1$, and $m'\ge1$, and for $H\in\Z^{m'\times m}$, let
\[
\Lambda_D^\perp(H)=\{y\in\Z^m: Hy\equiv\0\pmod D\}.
\]
This is a $D$-ary lattice, that is, $D\Z^m\subseteq\Lambda_D^\perp(H)\subseteq\Z^m$.
For $h\ge1$, let $\mathcal B_h=\{-(2^{h-1}-1),\dots,2^{h-1}\}^m$, a box with $2^h$ integers in each coordinate.

\begin{proposition}\label{prop:satlattice}
Let $D\ge2$, $m\ge1$, and $m'\ge1$ be integers, let $H\in\{0,\dots,D-1\}^{m'\times m}$, let $M=\Lambda_D^\perp(H)$, and let $t\in\Z^m$.
Let $h\ge1$, let $p=(p_1,\dots,p_s)$, where $p_1,\dots,p_s\ge2$ are pairwise coprime integers with $p_1\cdots p_s\ge2^h$, and let $p_{\max}=\max_jp_j$.
In time polynomial in $m$, $m'$, $\sum_jp_j$, $\log D$, and the encoding size of $t$, we can compute a formula $F=F(p,h,H,D,t)$ whose clauses have at most three literals, with the following properties.

\textup{(a)} The satisfying assignments of $F$ correspond one-to-one to the vectors $\tilde y\in M$ with $t-\tilde y\in\mathcal B_h$.

\textup{(b)} If $\lambda_1(M)>m^{3/2}p_{\max}(2^h-1)/2$, then every rational subpolyhedron of $P(F)$ has Chv\'atal rank at most one, and $F$ has at most one satisfying assignment.

\textup{(c)} $F$ has $O\big(m(\sum_jp_j)\log(p_1\cdots p_s)+m'mh\log D\big)$ variables.
\end{proposition}

We call the formula $F(p,h,H,D,t)$ \emph{thin} if $\lambda_1(\Lambda_D^\perp(H))>m^{3/2}p_{\max}(2^h-1)/2$, that is, if it satisfies the assumption of Proposition~\ref{prop:satlattice}(b).
Sections~\ref{sec:exactone}--\ref{sec:thinformulas} prove Proposition~\ref{prop:satlattice}, and Theorems~\ref{thm:satbdd} and~\ref{thm:satlwe} use only Proposition~\ref{prop:satlattice} from this section.

\subsection{Exact-one systems}\label{sec:exactone}

In this subsection we show that SAT polytopes can express systems of equations that set sums of literals equal to one, and that this preserves the promise (Lemma~\ref{lem:exactone}).
We first show that adding variables that are integral affine functions of the others commutes with taking integer hulls and Chv\'atal closures.

\begin{lemma}\label{lem:facelift}
Let $P\subseteq\R^N$ be a rational polyhedron, let $E\in\Z^{N'\times N}$, let $f\in\Z^{N'}$, and let $\varphi(x)=(x,Ex+f)$.
Then $\varphi$ maps $P\cap\Z^N$ bijectively onto $\varphi(P)\cap\Z^{N+N'}$, $\varphi(P)_I=\varphi(P_I)$, and $\varphi(P)'=\varphi(P')$.
\end{lemma}
\begin{proof}
We may assume $P\neq\emptyset$.
The map $\varphi$ is injective, and a vector $\varphi(x)$ is integral if and only if $x$ is integral.
Hence $\varphi$ maps $P\cap\Z^N$ bijectively onto $\varphi(P)\cap\Z^{N+N'}$, and so $\varphi(P)_I=\varphi(P_I)$.

We now show that $\varphi(P)'=\varphi(P')$.
Since $\varphi(P)$ is a rational polyhedron, $\varphi(P)'\subseteq\varphi(P)$, so every vector of $\varphi(P)'$ is $\varphi(x)$ for some $x\in P$.
Let $a=(a_1,a_2)\in\Z^{N+N'}$ and $c=a_1+E\T a_2\in\Z^N$.
Then $a\T\varphi(x)=c\T x+a_2\T f$ and $\delta_{\varphi(P)}(a)=\delta_P(c)+a_2\T f$, where $a_2\T f$ is an integer.
Hence the CG cut for $\varphi(P)$ along $a$ holds at $\varphi(x)$ if and only if the CG cut for $P$ along $c$ holds at $x$.
Since every $c\in\Z^N$ arises in this way, for instance with $a_2=\0$, the vector $\varphi(x)$ lies in $\varphi(P)'$ if and only if $x$ lies in $P'$.
\end{proof}

For a family $\mathcal T$ of nonempty sets of literals on variables $x_1,\dots,x_N$, let
\[
X(\mathcal T)=\Big\{x\in[0,1]^N:\ \sum_{\ell\in T}\ell(x)=1 \text{ for every } T\in\mathcal T\Big\}.
\]
We call the equations of $X(\mathcal T)$ \emph{exact-one equations}, and $X(\mathcal T)$ an \emph{exact-one system}.

\begin{lemma}\label{lem:exactone}
Let $\mathcal T$ be a family of nonempty sets of literals on variables $x_1,\dots,x_N$.
In time polynomial in $N+\sum_{T\in\mathcal T}|T|$, we can compute a formula $F$ with at most $N+\sum_{T\in\mathcal T}|T|$ variables, whose clauses have at most three literals, and whose satisfying assignments correspond one-to-one to the integral vectors of $X(\mathcal T)$.
If every rational subpolyhedron of $X(\mathcal T)$ has Chv\'atal rank at most one, then so does every rational subpolyhedron of $P(F)$.
\end{lemma}
\begin{proof}
We first construct a family $\mathcal T_3$ of sets of at most three literals, on the variables $x_1,\dots,x_N$ and $N'$ new variables, and a map $\varphi(x)=(x,Ex+f)$ with $E\in\Z^{N'\times N}$ and $f\in\Z^{N'}$, such that $X(\mathcal T_3)=\varphi(X(\mathcal T))$.
Let $T=\{\ell_1,\dots,\ell_\tau\}\in\mathcal T$ with $\tau\ge4$.
We introduce new variables $x'_1,\dots,x'_{\tau-3}$ and replace $T$ by the $\tau-2$ sets
\[
\{\ell_1,\ell_2,x'_1\},\ \{\bar x'_1,\ell_3,x'_2\},\ \dots,\ \{\bar x'_{\tau-4},\ell_{\tau-2},x'_{\tau-3}\},\ \{\bar x'_{\tau-3},\ell_{\tau-1},\ell_\tau\}.
\]
A vector satisfies the new equations if and only if $x'_j=1-\ell_1-\dots-\ell_{j+1}$ for $j=1,\dots,\tau-3$ and $\ell_1+\dots+\ell_\tau=1$.
Moreover, if the $\ell_i$ are nonnegative and sum to one, then these values of $x'_j$ lie in $[0,1]$.
Doing this for every set with more than three literals gives $\mathcal T_3$.
We let $\varphi$ append to $x$ the values of the new variables $x'_j$.
Since each $x'_j$ is an integral affine function of $x$, the map $\varphi$ has this form, and $X(\mathcal T_3)=\varphi(X(\mathcal T))$.

Next we encode $\mathcal T_3$ by a formula and prove the first statement.
Let $F$ consist of the clause $\{\ell_1\}$ for every $\{\ell_1\}\in\mathcal T_3$, the clauses $\{\ell_1,\ell_2\}$ and $\{\bar\ell_1,\bar\ell_2\}$ for every $\{\ell_1,\ell_2\}\in\mathcal T_3$, and the clauses $\{\ell_1,\ell_2,\ell_3\}$, $\{\bar\ell_1,\bar\ell_2\}$, $\{\bar\ell_1,\bar\ell_3\}$, and $\{\bar\ell_2,\bar\ell_3\}$ for every $\{\ell_1,\ell_2,\ell_3\}\in\mathcal T_3$.
Let $S$ be the set of vectors of $P(F)$ that satisfy $\ell_1+\ell_2+\ell_3=1$ for every $\{\ell_1,\ell_2,\ell_3\}\in\mathcal T_3$.
Then $S=X(\mathcal T_3)$.
Indeed, the clauses of a set with one or two literals hold together if and only if its literals sum to one.
For a set $\{\ell_1,\ell_2,\ell_3\}$, if $\ell_1+\ell_2+\ell_3=1$ and the literals are nonnegative, then $\ell_1+\ell_2\le1$, $\ell_1+\ell_3\le1$, and $\ell_2+\ell_3\le1$, which are the inequalities of the other three clauses.
Every integral vector of $P(F)$ lies in $S$, since there the clauses of each set $\{\ell_1,\ell_2,\ell_3\}$ force exactly one of its literals to be one.
So, by Lemma~\ref{lem:facelift}, the satisfying assignments of $F$ are the vectors $\varphi(x)$ with $x$ an integral vector of $X(\mathcal T)$, and so they correspond one-to-one to these vectors.
Since $\mathcal T_3$ has at most $\sum_{T\in\mathcal T}|T|$ sets, each giving at most four clauses, $F$ can be computed in time polynomial in $N+\sum_{T\in\mathcal T}|T|$.
Moreover, $F$ has $N+N'\le N+\sum_{T\in\mathcal T}|T|$ variables, since a set with $\tau\ge4$ literals adds $\tau-3$ new variables.

We now prove the second statement.
Let $R$ be a rational subpolyhedron of $P(F)$, and assume that every rational subpolyhedron of $X(\mathcal T)$ has Chv\'atal rank at most one.
We show that $R'$ lies in the face $R\cap S$ of $R$, and that this face is the image under $\varphi$ of a rational subpolyhedron of $X(\mathcal T)$, to which the assumption applies.
For every $\{\ell_1,\ell_2,\ell_3\}\in\mathcal T_3$, the three inequalities $\ell_1+\ell_2\le1$, $\ell_1+\ell_3\le1$, and $\ell_2+\ell_3\le1$ are valid for $P(F)$, and hence for $R$, and their sum is $2(\ell_1+\ell_2+\ell_3)\le3$.
Since $\ell_1+\ell_2+\ell_3$ is an affine function with integral coefficients and an integral constant term, this gives the CG cut $\ell_1+\ell_2+\ell_3\le1$ for $R$.
Hence $R'\subseteq R\cap S$.
Since $\ell_1+\ell_2+\ell_3\ge1$ is valid for $R$ for every such set, the set $R\cap S$ is a face of $R$.
The Chv\'atal closure of a face of a rational polyhedron is the intersection of the face with the Chv\'atal closure of the polyhedron \cite[Section~23.1]{Schrijver86}, so $(R\cap S)'=R'\cap S=R'$.
Every integral vector of $R$ lies in $S$, so $R_I=(R\cap S)_I$.
Since $R\cap S\subseteq X(\mathcal T_3)=\varphi(X(\mathcal T))$, we have $R\cap S=\varphi(\tilde R)$, where $\tilde R=\{x:\varphi(x)\in R\cap S\}$ is a rational subpolyhedron of $X(\mathcal T)$.
By Lemma~\ref{lem:facelift} and our assumption, $(R\cap S)'=\varphi(\tilde R')=\varphi(\tilde R_I)=(R\cap S)_I$.
Hence $R'=R_I$.
\end{proof}

\subsection{Binary arithmetic}\label{sec:circuits}

We now model arithmetic on binary numbers by exact-one systems.
We introduce constraints that set some $0/1$ variables equal to a given function of earlier ones, and we write each of them with the standard extended formulation of the convex hull of its solutions.
With them, we compute weighted sums modulo an integer in binary (Lemma~\ref{lem:arith}), which Section~\ref{sec:residue} uses to recover the entries from their residues and to check that a vector lies in the lattice.

We define \emph{functional formulations}, the systems built from these constraints.
We write $x_u$ for the variable with index $u$, and we also call it the variable $u$.
A functional formulation $X$ starts from \emph{free} variables, with \emph{multiple-choice constraints} $\sum_{u\in I}x_u=1$ for some disjoint sets $I$ of free variables, and adds a finite sequence of functional constraints and fixings, as follows.
A \emph{functional constraint} $g$ introduces new variables, its \emph{outputs}, as a function of at most three earlier variables, its \emph{inputs}, that is, of free variables or variables introduced by earlier functional constraints.
Writing $I_g$ and $O_g$ for the sequences of its inputs and outputs, and $x_{I_g}$ and $x_{O_g}$ for the corresponding vectors of variables, it requires $x_{O_g}=f_g(x_{I_g})$ for a given function $f_g:\{0,1\}^{|I_g|}\to\{0,1\}^{|O_g|}$.
Its \emph{solutions} are the $2^{|I_g|}$ vectors $(v,f_g(v))$ with $v\in\{0,1\}^{|I_g|}$.
We write it with the standard extended formulation of the convex hull of its solutions, that is, with new variables $x_{g,v}$ for $v\in\{0,1\}^{|I_g|}$ and the constraints
\begin{equation}\label{eq:extform}
x_{g,v}\ge0,\qquad\sum_vx_{g,v}=1,\qquad(x_{I_g},x_{O_g})=\sum_vx_{g,v}\,(v,f_g(v)).
\end{equation}
For instance, for a functional constraint $g$ with inputs $u,w$, output $o$, and $f_g(v_1,v_2)=v_1v_2$, the solutions are $(0,0,0)$, $(0,1,0)$, $(1,0,0)$, and $(1,1,1)$, and the last equation in \eqref{eq:extform} reads $x_u=x_{g,(1,0)}+x_{g,(1,1)}$, $x_w=x_{g,(0,1)}+x_{g,(1,1)}$, and $x_o=x_{g,(1,1)}$.
Finally, some outputs are \emph{fixed} to $0$ or $1$.
The variables of $X$ are its free variables, the outputs, and the variables $x_{g,v}$.
The functional formulation $X$ is the system
\[
\begin{aligned}
&\textstyle\sum_{u\in I}x_u=1&&\text{for every set }I\text{ of a multiple-choice constraint},\\
&\text{\eqref{eq:extform}}&&\text{for every functional constraint }g,\\
&x_u=\epsilon_u&&\text{for every output }u\text{ fixed to }\epsilon_u\in\{0,1\},\\
&0\le x_u\le1&&\text{for every free variable }u,
\end{aligned}
\]
and we also denote by $X$ the polyhedron it defines.
All these equations are exact-one equations in the literals $x_u$ and $\bar x_u=1-x_u$.
Indeed, each entry $x_u$ of $(x_{I_g},x_{O_g})$ gives, in the last equation in \eqref{eq:extform}, an equation $x_u=\sum_{v:u(v)=1}x_{g,v}$, which we write as $\bar x_u+\sum_{v:u(v)=1}x_{g,v}=1$, where $u(v)$ is the corresponding entry of $(v,f_g(v))$.
We write fixing $u$ to $1$ or $0$ as $x_u=1$ or $\bar x_u=1$, respectively.
Thus, since $0\le x_u\le1$ for every variable $u$ and every $x\in X$, the functional formulation $X$ is an exact-one system, that is, $X=X(\mathcal T)$ as in Section~\ref{sec:exactone}, where $\mathcal T$ is the family of the sets of literals of these equations.
For $x\in X$, we denote by $b(x)$ the vector of its free variables.
An \emph{integral linear function} of the variables of $X$ is a linear combination of these variables with integer coefficients.
For a polyhedron $P$ and an affine function $\psi$, we write $\delta_P(\psi)=\sup\{\psi(x):x\in P\}$ and $\width_P(\psi)=\delta_P(\psi)+\delta_P(-\psi)$.

The next lemma gives the two properties of functional formulations that we use.
Part (a) is the reason for using extended formulations: for a general formulation of the same integral vectors, such as one given by clauses, an affine identity that holds at the solutions of a functional constraint can fail at fractional vectors.

\begin{lemma}\label{lem:circuit}
Let $X$ be a functional formulation.

\textup{(a)} Let $g$ be a functional constraint of $X$, and let $\eta$ be an affine function that vanishes at the solutions $(v,f_g(v))$ of $g$. Then $\eta(x_{I_g},x_{O_g})=0$ for every $x\in X$.

\textup{(b)} Let $x\in X$ be such that $b(x)$ is integral. Then $x$ is integral, and it is the only vector of $X$ with free variables $b(x)$.
\end{lemma}
\begin{proof}
(a)
For every $x\in X$, the vector $(x_{I_g},x_{O_g})$ is a convex combination of the solutions of $g$, and $\eta$ vanishes at them, and hence at this convex combination, since $\eta$ is affine.

(b)
Going through the functional constraints in order, suppose that the inputs of $g$ take $0/1$ values $v^*$ in $x$, which holds for free inputs by assumption and for the other inputs by induction.
Then $v^*=\sum_vx_{g,v}v$, and since $v^*$ is a vertex of $[0,1]^{|I_g|}$, we obtain $x_{g,v^*}=1$ and $x_{g,v}=0$ for $v\ne v^*$, and so $x_{O_g}=f_g(v^*)$.
Hence all variables of $x$ are integral and determined by $b(x)$.
\end{proof}

The next lemma computes in binary the residue of a weighted sum of variables modulo an integer, and writes the quotient as an integral linear function.

\begin{lemma}\label{lem:arith}
Let $X_0$ be a functional formulation, let $k\ge1$, let $u_1,\dots,u_k$ be variables of $X_0$ with integer weights $\alpha_1,\dots,\alpha_k\ge0$, let $K\ge1$ be an integer, and let $L$ be the smallest integer with $2^L\ge2K$.
In time polynomial in $k$, $\log K$, and $\log(1+\max_i\alpha_i)$, we can add to $X_0$ new functional constraints with at most three inputs and at most two outputs each, $O(kL)$ in total, and new fixings, obtaining a functional formulation $X$ with the same free variables and multiple-choice constraints, outputs $\rho_0,\dots,\rho_{L-1}$ of $X$ for the residue, and an integral linear function $\kappa$ of the variables of $X$ for the quotient, with the following properties.

\textup{(a)} We have $\sum_i\alpha_ix_{u_i}=K\kappa(x)+\sum_l2^lx_{\rho_l}$ for every $x\in X$.

\textup{(b)} Every integral vector of $X_0$ extends to exactly one integral vector of $X$, and in it $\rho_0,\dots,\rho_{L-1}$ take the binary digits of $\big(\sum_i\alpha_ix_{u_i}\big)\bmod K$.
\end{lemma}
\begin{proof}
We may assume that $\alpha_i<K$ for every $i$.
Indeed, writing $\alpha_i=K\fl{\alpha_i/K}+\alpha'_i$ with $\alpha'_i\in\{0,\dots,K-1\}$, we have $\sum_i\alpha_ix_{u_i}=\sum_i\alpha'_ix_{u_i}+K\sum_i\fl{\alpha_i/K}\,x_{u_i}$.
Since $\alpha'_i\equiv\alpha_i\pmod K$, the residue in (b) is the same for both weights.
So it suffices to prove the lemma for the weights $\alpha'_i$, and then to add $\sum_i\fl{\alpha_i/K}\,x_{u_i}$ to $\kappa$.

We first describe how to add a fixed multiple of a variable to a number whose binary digits are variables, and how to subtract it.
Here, for a sequence $r=(r_0,\dots,r_{L-1})$ of outputs, the number is $r(x)=\sum_l2^lx_{r_l}$, and we use the same notation for every such sequence.
The multiple is $\theta x_u$ for a variable $u$ of $X$ and an integer $\theta\in\{0,\dots,2^L-1\}$ with binary digits $\theta_0,\dots,\theta_{L-1}$.
Some inputs of these operations are constant.
To treat them like the other inputs, we obtain the constants $0$ and $1$ as the outputs of two functional constraints without inputs.
To add $\theta$ times $u$ to $r$, that is, to obtain a sequence $s=(s_0,\dots,s_{L-1})$ of outputs with $s(x)=r(x)+\theta x_u$, we introduce, for $l=0,\dots,L-1$, a functional constraint with inputs $r_l$, $u$, and $o_l$, and with outputs $s_l$ and $o_{l+1}$, the \emph{digit} and the \emph{carry}.
Its function maps the values $\hat x_{r_l},\hat x_u,\hat x_{o_l}\in\{0,1\}$ of the inputs to the values $\hat x_{s_l},\hat x_{o_{l+1}}\in\{0,1\}$ with $\hat x_{r_l}+\theta_l\hat x_u+\hat x_{o_l}=\hat x_{s_l}+2\hat x_{o_{l+1}}$.
The first carry $o_0$ is the constant $0$, and we fix the last carry $o_L$ to $0$.
By Lemma~\ref{lem:circuit}(a), we have $x_{r_l}+\theta_lx_u+x_{o_l}=x_{s_l}+2x_{o_{l+1}}$ for every $x\in X$ and every $l$, and multiplying by $2^l$ and summing over $l$ gives $s(x)=r(x)+\theta x_u$, since $x_{o_0}=0$ and $o_L$ is fixed to $0$.
Let $\hat x$ be a $0/1$ vector on the variables of $X$ in which each output takes the value of its function at the values of its inputs, and in which $r$ takes the binary digits of an integer $z$.
If $z+\theta\hat x_u<2^L$, then, as in schoolbook addition, $s$ takes the binary digits of $z+\theta\hat x_u$ and $\hat x_{o_L}=0$.
To subtract $\theta$ times $u$ from $r$, that is, to obtain $s$ with $s(x)=r(x)-\theta x_u$, we do the same, but with $\hat x_{r_l}-\theta_l\hat x_u-\hat x_{o_l}=\hat x_{s_l}-2\hat x_{o_{l+1}}$.
As with additions, $s(x)=r(x)-\theta x_u$ for every $x\in X$.
For $\hat x$ and $z$ as above, if $z-\theta\hat x_u\ge0$, then, as in schoolbook subtraction, $s$ takes the binary digits of $z-\theta\hat x_u$ and $\hat x_{o_L}=0$.

We now use these operations to compute the residue, by adding the terms $\alpha_ix_{u_i}$ one at a time and reducing modulo $K$ after each addition.
Let every entry of $\rho^0$ be the constant $0$.
For $i=1,\dots,k$, we obtain $\rho^i$ from $\rho^{i-1}$ in four steps.
We add $\alpha_i$ times $u_i$ to $\rho^{i-1}$, with result $s^i$.
To compare $s^i(x)$ with $K$, we subtract $K$ times the constant $1$ from $s^i$.
We use only the last carry of this subtraction, which we call $\beta_i$ and do not fix.
For integral $x\in X$, it takes the value $1$ exactly when $s^i(x)<K$.
We introduce a functional constraint with input $\beta_i$ and output $q_i$, with $\hat x_{q_i}=1-\hat x_{\beta_i}$, so that $q_i$ indicates whether we must subtract $K$.
Finally, we subtract $K$ times $q_i$ from $s^i$, with result $\rho^i$.
We let $\rho=\rho^k$ and $\kappa(x)=\sum_ix_{q_i}$.
Each of the $3k$ additions and subtractions introduces $L$ functional constraints, and there are $k+2$ others, so there are $O(kL)$ new functional constraints, and they can be computed in the stated time.

We now prove (a).
By the identities for additions and subtractions, we have $s^i(x)=\rho^{i-1}(x)+\alpha_ix_{u_i}$ and $\rho^i(x)=s^i(x)-Kx_{q_i}$ for every $x\in X$ and every $i$, and summing over $i$ gives (a), since $\rho^0(x)=0$ and $\rho=\rho^k$.

We now prove (b).
Let $x^0$ be an integral vector of $X_0$.
We extend it to a $0/1$ vector $\hat x$ on the variables of $X$: each new output takes the value of its function at the values of its inputs, one functional constraint after the other, and each new variable $x_{g,v}$ takes the value $1$ if $v$ is the vector of values of the inputs of $g$, and $0$ otherwise.
Let $z_i=\big(\sum_{i'\le i}\alpha_{i'}\hat x_{u_{i'}}\big)\bmod K$.
We show by induction on $i$ that, in $\hat x$, $\rho^i$ takes the binary digits of $z_i$ and the fixings introduced for the first $i$ terms are satisfied.
This holds for $i=0$.
Let $i\ge1$.
Since $z_{i-1}<K$ and $\alpha_i<K$, we have $z_{i-1}+\alpha_i\hat x_{u_i}<2K\le2^L$, so, as shown above for additions, $s^i$ takes the binary digits of $z_{i-1}+\alpha_i\hat x_{u_i}$, and the last carry of this addition, which is fixed to $0$, takes the value $0$.
Then, as in schoolbook subtraction, $\beta_i$ takes the value $1$ if and only if $z_{i-1}+\alpha_i\hat x_{u_i}<K$, so $q_i$ takes the value $1$ if and only if $z_{i-1}+\alpha_i\hat x_{u_i}\ge K$.
Hence $z_{i-1}+\alpha_i\hat x_{u_i}-K\hat x_{q_i}$ lies in $\{0,\dots,K-1\}$ and is congruent to $z_i$ modulo $K$, so it equals $z_i$.
Since it is nonnegative, as shown above for subtractions, $\rho^i$ takes its binary digits, and the last carry of this subtraction takes the value $0$.
For $i=k$, this shows that $\hat x$ satisfies the new fixings.
It also satisfies the extended formulations of the new functional constraints, by construction, and the constraints of $X_0$, since $x^0\in X_0$, so $\hat x\in X$.
By Lemma~\ref{lem:circuit}(b), $\hat x$ is the only vector of $X$ with free variables $b(x^0)$, and so the only integral vector of $X$ that extends $x^0$.
In it, $\rho=\rho^k$ takes the binary digits of $z_k$, which is $\big(\sum_i\alpha_ix^0_{u_i}\big)\bmod K$.
\end{proof}

\subsection{Residue encodings}\label{sec:residue}

We now construct the functional formulations from which we obtain the formulas $F(p,h,H,D,t)$ of Proposition~\ref{prop:satlattice}.
We encode vectors with possibly exponentially large entries by their residues.
Let $p=(p_1,\dots,p_s)$, where $p_1,\dots,p_s\ge2$ are pairwise coprime integers, let $\Pi=p_1\cdots p_s$, and let $p_{\max}=\max_jp_j$.
An integer in $\{0,\dots,\Pi-1\}$ is determined by its residues modulo $p_1,\dots,p_s$, and we encode the residue modulo $p_j$ by $p_j$ free variables with a multiple-choice constraint.
For each $j$, let $\iota_j\in\{1,\dots,\Pi-1\}$ with $\iota_j\equiv1\pmod{p_j}$ and $\iota_j\equiv0\pmod{p_{j'}}$ for $j'\ne j$.
By the Chinese remainder theorem, the integer in $\{0,\dots,\Pi-1\}$ with residues $a_1,\dots,a_s$ modulo $p_1,\dots,p_s$ is $\big(\sum_j\iota_ja_j\big)\bmod\Pi$.
We compute it in binary with Lemma~\ref{lem:arith}, so that, for every vector satisfying the resulting formulation, the number written by its binary digits equals a linear function of the free variables minus $\Pi$ times an integral linear function.

Let $h\ge1$ with $2^h\le\Pi$, let $D\ge2$, $m\ge1$, and $m'\ge1$ be integers, let $H\in\{0,\dots,D-1\}^{m'\times m}$, let $M=\Lambda_D^\perp(H)$, and let $t\in\Z^m$.
Let $d\in\{0,\dots,D-1\}^{m'}$ with $d\equiv Ht\pmod D$.
The functional formulation $X=X(p,h,H,D,t)$ is built in three parts.

First, for $i=1,\dots,m$ and $j=1,\dots,s$, it has free variables $x_{i,j,a}$ for $a=0,\dots,p_j-1$, with the multiple-choice constraint $\sum_ax_{i,j,a}=1$.
A $0/1$ vector $x$ satisfying these constraints \emph{encodes} the vector in $\{0,\dots,\Pi-1\}^m$ whose $i$th entry has residue $a$ modulo $p_j$ whenever $x_{i,j,a}=1$.
For instance, for $p=(3,4)$, if $x_{i,1,1}=x_{i,2,3}=1$, then the $i$th entry of the encoded vector is $7$, since $7\equiv1\pmod3$ and $7\equiv3\pmod4$.

Second, for $i=1,\dots,m$ in turn, we apply Lemma~\ref{lem:arith} to the current formulation, with the variables $x_{i,j,a}$ for all $j$ and $a$, the weight $\iota_ja$ for $x_{i,j,a}$, and $K=\Pi$.
This gives outputs $\rho_{i,0},\rho_{i,1},\dots$ and an integral linear function $\kappa_i$.
We fix the outputs $\rho_{i,l}$ with $l\ge h$ to $0$, and we let $c_i(x)=\sum_{l<h}2^lx_{\rho_{i,l}}$.
For integral $x$, these fixings force the encoded vector to lie in $\{0,\dots,2^h-1\}^m$, and hence to equal $c(x)=(c_1(x),\dots,c_m(x))$.

Third, for $r=1,\dots,m'$ in turn, we apply Lemma~\ref{lem:arith} to the current formulation, with the variables $\rho_{i,l}$ for $i=1,\dots,m$ and $l<h$, the weight $H'_{r,(i,l)}$ for $\rho_{i,l}$, where $H'_{r,(i,l)}\in\{0,\dots,D-1\}$ and $H'_{r,(i,l)}\equiv2^lH_{ri}\pmod D$, and $K=D$.
This gives outputs $\sigma_{r,0},\sigma_{r,1},\dots$, which we fix to the binary digits of $d_r$, and an integral linear function $\kappa'_r$.
For integral $x$, these fixings force $Hc(x)\equiv Ht\pmod D$, that is, $c(x)\in t+M$.

The next lemma gives the properties of $X(p,h,H,D,t)$ that we use.

\begin{lemma}\label{lem:crtcircuit}
Let $X=X(p,h,H,D,t)$ and $c(x)=(c_1(x),\dots,c_m(x))$, and let $V_i(x)=\sum_{j,a}\iota_ja\,x_{i,j,a}$ for every $i$.

\textup{(a)} We have $c_i(x)=V_i(x)-\Pi\kappa_i(x)$ and $0\le c_i(x)\le2^h-1$ for every $x\in X$ and every $i$. Moreover, there are integral linear functions $\zeta_1,\dots,\zeta_{m'}$ of the variables of $X$ with $(Hc(x)-d)/D=(\zeta_1(x),\dots,\zeta_{m'}(x))$ for every $x\in X$.

\textup{(b)} The integral vectors of $X$ correspond one-to-one to the vectors in $(t+M)\cap\{0,\dots,2^h-1\}^m$.

\textup{(c)} We have $X=X(\mathcal T)$ for a family $\mathcal T$ of sets of at most $\max\{9,p_{\max}\}$ literals, and $X$ has $O\big(m(\sum_jp_j)\log\Pi+m'mh\log D\big)$ functional constraints. Moreover, $\mathcal T$ can be computed in time polynomial in $m$, $m'$, $\sum_jp_j$, $\log D$, and the encoding size of $t$.
\end{lemma}
\begin{proof}
Let $X_1$ and $X_2$ be the functional formulations obtained after the first and the second part of the construction of $X$, so that $X$ is obtained from $X_2$ by the third part.
Each part only adds variables and constraints, so every vector of $X$ restricts to vectors of $X_1$ and $X_2$.
Let $\rho$ be the family of the outputs $\rho_{i,l}$ with $l<h$, indexed by the pairs $(i,l)$, and let $G\in\Z^{m\times mh}$ have entry $2^l$ in row $i$ and column $(i,l)$, and $0$ elsewhere, so that $c(x)=Gx_\rho$ and $H'\equiv HG\pmod D$.

(a)
By Lemma~\ref{lem:arith}(a), $V_i(x)=\Pi\kappa_i(x)+\sum_l2^lx_{\rho_{i,l}}$ for every $x\in X_2$, and hence for every $x\in X$.
Since the outputs $\rho_{i,l}$ with $l\ge h$ are fixed to $0$, we obtain $V_i(x)=\Pi\kappa_i(x)+c_i(x)$ for every $x\in X$.
Moreover, $0\le c_i(x)\le2^h-1$, since all variables lie in $[0,1]$.
By Lemma~\ref{lem:arith}(a), $H'_rx_\rho=D\kappa'_r(x)+\sum_l2^lx_{\sigma_{r,l}}$ for every $x\in X$ and every $r$, where $H'_r$ is the $r$th row of $H'$.
The outputs $\sigma_{r,l}$ are fixed to the binary digits of $d_r$, so $H'_rx_\rho=D\kappa'_r(x)+d_r$ for every $x\in X$.
Writing $HG=H'+DW$ with $W$ integral, we obtain $H_rc(x)=H_rGx_\rho=d_r+D(\kappa'_r(x)+W_rx_\rho)$, where $H_r$ and $W_r$ are the $r$th rows of $H$ and $W$.
Hence, for every $x\in X$, the $r$th entry of $(Hc(x)-d)/D$ equals $\zeta_r(x)=\kappa'_r(x)+W_rx_\rho$, which is an integral linear function, since $\kappa'_r$ is an integral linear function and $W$ is integral.

(b)
The integral vectors of $X_1$ are the $0/1$ vectors that satisfy the multiple-choice constraints.
Such a vector has exactly one $a=a_{ij}$ with $x_{i,j,a}=1$ for all $i$ and $j$, and it encodes the vector $c\in\{0,\dots,\Pi-1\}^m$ with $c_i\equiv a_{ij}\pmod{p_j}$ for all $j$.
By the Chinese remainder theorem, this gives a bijection from the integral vectors of $X_1$ to $\{0,\dots,\Pi-1\}^m$.
Since $\sum_j\iota_ja_{ij}\equiv c_i\pmod\Pi$, Lemma~\ref{lem:arith}(b), applied for $i=1,\dots,m$, shows that each integral vector of $X_1$ extends to exactly one integral vector satisfying the formulation obtained from $X_1$ by the applications of Lemma~\ref{lem:arith} in the second part, and in it the outputs $\rho_{i,l}$ take the binary digits of $c_i$.
This vector satisfies the fixings of the second part if and only if $c\in\{0,\dots,2^h-1\}^m$.
Hence the integral vectors of $X_1$ with $c$ in this box extend to exactly one integral vector of $X_2$ each, in which the outputs $\rho_{i,l}$ with $l<h$ take the binary digits of $c_i$, and the others extend to none.
By Lemma~\ref{lem:arith}(b), applied for $r=1,\dots,m'$, each integral vector of $X_2$ extends to exactly one integral vector satisfying the formulation obtained from $X_2$ by the applications of Lemma~\ref{lem:arith} in the third part, and in it the outputs $\sigma_{r,l}$ take the binary digits of $(H_rc)\bmod D$, since $H'\equiv HG\pmod D$.
This vector satisfies the fixings of the third part, and so lies in $X$, if and only if $Hc\equiv d\equiv Ht\pmod D$, that is, $c\in t+M$.
Since every integral vector of $X$ restricts to integral vectors of $X_1$ and $X_2$, the integral vectors of $X$ correspond one-to-one to the vectors $c\in(t+M)\cap\{0,\dots,2^h-1\}^m$, which lie in $\{0,\dots,\Pi-1\}^m$ since $2^h\le\Pi$.

(c)
The functional constraints have at most three inputs, so the exact-one equations of their extended formulations have at most $9$ literals: $\sum_vx_{g,v}=1$ has at most $8$, and each coordinate equation has at most $8$ variables $x_{g,v}$ and one more literal.
The multiple-choice constraints have $p_j$ literals.
The second part applies Lemma~\ref{lem:arith} $m$ times with $k=\sum_jp_j$ and $K=\Pi$, and the third part applies it $m'$ times with $k=mh$ and $K=D$.
Since each application adds $O(kL)$ functional constraints, where $L=O(\log K)$, this gives the stated number of functional constraints.
This number is polynomial in $m$, $m'$, $\sum_jp_j$, and $\log D$, since $h\le\log_2\Pi\le\sum_jp_j$.
The numbers $\iota_j$ can be computed by the extended Euclidean algorithm, and $H'$ and $d$ by reducing $2^lH_{ri}$ and $Ht$ modulo $D$.
\end{proof}

\subsection{Thin formulas}\label{sec:thinformulas}

We now show that thin formulas satisfy the promise, which completes the proof of Proposition~\ref{prop:satlattice}.
We use the following description of the dual lattice of $\Lambda_D^\perp(H)$, which is stated in \cite[Section~2]{MicciancioRegev09}, and we include its short proof.

\begin{lemma}\label{lem:qarydual}
Let $D\ge1$, $m\ge1$, and $m'\ge1$ be integers, let $H\in\Z^{m'\times m}$, and let $M=\Lambda_D^\perp(H)$.
Then $DM^*=H\T\Z^{m'}+D\Z^m$.
\end{lemma}
\begin{proof}
Let $L_0=H\T\Z^{m'}+D\Z^m$, and let $L_0^*$ be its dual lattice.
A vector $y\in\R^m$ lies in $DL_0^*$ if and only if $v\T y\in D\Z$ for every $v\in L_0$.
Since this condition is preserved under integer combinations of $v$, it suffices to check it for the vectors $De_i$ with $i=1,\dots,m$ and $H\T z$ with $z\in\Z^{m'}$, which generate $L_0$.
For the vectors $De_i$, this condition holds if and only if $y\in\Z^m$, and for the vectors $H\T z$, it holds if and only if $Hy\in D\Z^{m'}$.
Hence $DL_0^*=M$.
Since $(DL)^*=D^{-1}L^*$ for every lattice $L$, we obtain $DM^*=D(DL_0^*)^*=L_0^{**}=L_0$.
\end{proof}

The next lemma plays the role of Lemma~\ref{lem:dualbasis}: instead of short vectors of the dual lattice, we need vectors of the dual lattice close to given vectors.

\begin{lemma}\label{lem:dualcover}
Let $M\subseteq\R^m$ be a lattice.
For every $y\in\R^m$, there is $w\in M^*$ with $\|w-y\|_1\le m^{3/2}/(2\lambda_1(M))$.
\end{lemma}
\begin{proof}
By a transference theorem of Banaszczyk \cite[Theorem~2.2]{Banaszczyk93}, $\operatorname{dist}(y,M^*)\le m/(2\lambda_1(M))$.
Let $w\in M^*$ with $\|w-y\|=\operatorname{dist}(y,M^*)$.
Then $\|w-y\|_1\le\sqrt m\|w-y\|\le m^{3/2}/(2\lambda_1(M))$.
\end{proof}

The next proposition is the main step in the proof of Proposition~\ref{prop:satlattice}.
It plays the role of Proposition~\ref{cor:promise} for the functional formulations.

\begin{proposition}\label{thm:crt}
Let $X=X(p,h,H,D,t)$ and $M=\Lambda_D^\perp(H)$.
If $\lambda_1(M)>m^{3/2}p_{\max}(2^h-1)/2$, then every rational subpolyhedron of $X$ has Chv\'atal rank at most one.
\end{proposition}
\begin{proof}
Let $R$ be a rational subpolyhedron of $X$. We may assume $R\ne\emptyset$.
We find integral linear functions along which $R$ has width less than one, such that every free variable is the difference of two of them.
Then CG cuts along these functions fix all free variables, and hence, by Lemma~\ref{lem:circuit}(b), leave at most one integral vector.

We obtain these functions from affine functions, defined separately for each $i$ and $j$.
Fix $i$ and $j$.
Most objects defined below depend on $i$ and $j$, but until the last paragraph we do not show this in the notation.
By Lemma~\ref{lem:dualcover}, there is $w\in M^*$ with $\|w-e_i/p_j\|_1\le m^{3/2}/(2\lambda_1(M))<1/(p_j(2^h-1))$.
Let $w'=w-e_i/p_j$.
For $l=0,\dots,p_j$, let $f_l(a)=a/p_j$ for $a<l$ and $f_l(a)=a/p_j-1$ for $a\ge l$, and let
\[
\psi_l(x)=\sum_{a=0}^{p_j-1}f_l(a)\,x_{i,j,a}+w'^{\mathsf T}c(x).
\]
The functions $f_l$ have three properties: $f_l(a)-a/p_j$ is an integer, $f_{l+1}(a)-f_l(a)$ is $1$ if $a=l$ and $0$ otherwise, and the values of $f_l$ lie in an interval of length less than one.
For instance, for $p_j=3$, the vectors $(f_l(0),f_l(1),f_l(2))$ for $l=0,1,2,3$ are $(-1,-\frac23,-\frac13)$, $(0,-\frac23,-\frac13)$, $(0,\frac13,-\frac13)$, and $(0,\frac13,\frac23)$.
We use them, in turn, to show that, for every $x\in X$, each $\psi_l(x)$ equals an integral linear function of $x$ plus a constant, that consecutive differences of these integral linear functions are the free variables, and that $R$ has width less than one along them.

We first show that, for every $x\in X$, each $\psi_l(x)$ equals an integral linear function of $x$ plus a constant that does not depend on $l$.
By Lemma~\ref{lem:qarydual}, $Dw=H\T z+Dz'$ for some $z\in\Z^{m'}$ and $z'\in\Z^m$.
By Lemma~\ref{lem:crtcircuit}(a), for every $x\in X$, we have $w\T c(x)=z\T Hc(x)/D+z'^{\mathsf T}c(x)=z\T d/D+z\T\zeta(x)+z'^{\mathsf T}c(x)$, where $\zeta=(\zeta_1,\dots,\zeta_{m'})$, and $z\T\zeta+z'^{\mathsf T}c$ is an integral linear function, since the entries of $\zeta$ and of $c$ are.
By Lemma~\ref{lem:crtcircuit}(a) again, for every $x\in X$,
\[
\frac{c_i(x)}{p_j}=\frac{V_i(x)-\Pi\kappa_i(x)}{p_j}=\sum_{a}\frac{\iota_ja}{p_j}\,x_{i,j,a}+\sum_{j'\ne j}\sum_a\frac{\iota_{j'}a}{p_j}\,x_{i,j',a}-\frac \Pi{p_j}\,\kappa_i(x).
\]
We have $\iota_{j'}\equiv0\pmod{p_j}$ for $j'\ne j$, the integer $p_j$ divides $\Pi$, and $\iota_j-1$ is a multiple of $p_j$.
Hence $c_i(x)/p_j=\sum_a(a/p_j)\,x_{i,j,a}+\omega(x)$ for every $x\in X$, with $\omega$ an integral linear function.
Since $e_i\T c(x)=c_i(x)$, for every $x\in X$ we obtain
\[
\psi_l(x)=\sum_a\big(f_l(a)-a/p_j\big)\,x_{i,j,a}-\omega(x)+w\T c(x)=\psi'_l(x)+z\T d/D,
\]
where $\psi'_l$ is an integral linear function, since $f_l(a)-a/p_j\in\{0,-1\}$, and $z\T d/D$ does not depend on $l$.
Moreover, $\psi'_{l+1}-\psi'_l=x_{i,j,l}$ for $l=0,\dots,p_j-1$, by the second property of the functions $f_l$.

Next we show that $R$ has width less than one along each function $\psi'_l$.
For every $x\in X$, the vector $(x_{i,j,a})_a$ lies in the standard simplex, so the first term of $\psi_l(x)$ lies in $[\min_af_l(a),\max_af_l(a)]$.
This interval has length at most $1-1/p_j$.
Indeed, its endpoints are $l/p_j-1$ and $(l-1)/p_j$ for $1\le l\le p_j-1$, and they are $-1$ and $-1/p_j$, or $0$ and $(p_j-1)/p_j$, for $l=0$ or $l=p_j$, respectively.
By Lemma~\ref{lem:crtcircuit}(a), $c(x)\in[0,2^h-1]^m$ for every $x\in X$, so $w'^{\mathsf T}c(x)$ varies by at most $\|w'\|_1(2^h-1)<1/p_j$ as $x$ ranges over $X$.
Hence, since $R\subseteq X$ and $\psi_l(x)-\psi'_l(x)=z\T d/D$ for every $x\in X$, $\width_R(\psi'_l)\le\width_{X}(\psi'_l)=\width_{X}(\psi_l)<1$.

Finally, we show that CG cuts along these functions fix all free variables.
Let $R_0$ be the intersection of $R$ with the CG cuts $\psi'_{i,j,l}\le\fl{\delta_R(\psi'_{i,j,l})}$ and $\psi'_{i,j,l}\ge-\fl{\delta_R(-\psi'_{i,j,l})}$ for all $i$, $j$, and $l$.
If $R_0\ne\emptyset$, then every $\psi'_{i,j,l}$ takes the same integral value at every vector of $R_0$, since the bounds $\fl{\delta_R(\psi'_{i,j,l})}$ and $-\fl{\delta_R(-\psi'_{i,j,l})}$ differ by an integer at most $\width_R(\psi'_{i,j,l})<1$.
Hence every free variable $x_{i,j,l}=\psi'_{i,j,l+1}-\psi'_{i,j,l}$ takes the same integral value at every vector of $R_0$, that is, $0$ or $1$.
By Lemma~\ref{lem:circuit}(b), $R_0$ is empty or a single integral vector, which then lies in $R$.
Thus $R_0\subseteq R_I$.
Since also $R_I\subseteq R'\subseteq R_0$, these three sets are equal, so $R$ has Chv\'atal rank at most one.
\end{proof}

\begin{proof}[Proof of Proposition~\ref{prop:satlattice}]
Let $X=X(p,h,H,D,t+(2^{h-1}-1)\1)$, and let $F=F(p,h,H,D,t)$ be the formula given by Lemma~\ref{lem:exactone} for $X$.
By Lemmas~\ref{lem:exactone} and~\ref{lem:crtcircuit}(c), its clauses have at most three literals, and it can be computed in the stated time.

(a)
By Lemma~\ref{lem:exactone}, the satisfying assignments of $F$ correspond one-to-one to the integral vectors of $X$.
By Lemma~\ref{lem:crtcircuit}(b), these correspond one-to-one to the vectors $c\in(t+(2^{h-1}-1)\1+M)\cap\{0,\dots,2^h-1\}^m$.
Writing $c=t-\tilde y+(2^{h-1}-1)\1$, these are the vectors $\tilde y\in M$ with $t-\tilde y\in\mathcal B_h$.

(b)
By Proposition~\ref{thm:crt}, every rational subpolyhedron of $X$ has Chv\'atal rank at most one, and so, by Lemma~\ref{lem:exactone}, does every rational subpolyhedron of $P(F)$.
By (a), the satisfying assignments of $F$ correspond to the vectors $\tilde y\in M$ with $t-\tilde y\in\mathcal B_h$.
For two such vectors $\tilde y$ and $\tilde y'$, we have $\tilde y-\tilde y'\in M$ and $\|\tilde y-\tilde y'\|_\infty\le2^h-1$, since $\mathcal B_h$ has $2^h$ consecutive integers in each coordinate.
Hence, since $p_{\max}\ge2$, we have $\|\tilde y-\tilde y'\|\le\sqrt m(2^h-1)\le m^{3/2}p_{\max}(2^h-1)/2<\lambda_1(M)$, and so $\tilde y=\tilde y'$.

(c)
Let $\Pi=p_1\cdots p_s$, and let $\mathcal T$ be the family of Lemma~\ref{lem:crtcircuit}(c), so that $X=X(\mathcal T)$.
By Lemma~\ref{lem:crtcircuit}(c), $X$ has $O\big(m(\sum_jp_j)\log\Pi+m'mh\log D\big)$ functional constraints.
By Lemma~\ref{lem:arith}, each of them has at most three inputs and at most two outputs, so it adds at most ten variables, at most six exact-one equations of at most nine literals each, and at most two fixings.
The $m\sum_jp_j$ free variables lie in multiple-choice constraints with $m\sum_jp_j$ literals in total.
Hence, since $\log\Pi\ge\log2$, the number of variables of $X$ plus $\sum_{T\in\mathcal T}|T|$ is $O\big(m(\sum_jp_j)\log\Pi+m'mh\log D\big)$, and (c) follows from Lemma~\ref{lem:exactone}.
\end{proof}

\section{Hardness for SAT polytopes}\label{sec:sathard}

In this section we reduce worst-case lattice problems and learning with errors to satisfiability under the promise $P(F)'=P(F)_I$, using the formulas of Proposition~\ref{prop:satlattice}.
This proves Theorems~C and~D, through Theorems~\ref{thm:satbdd} and~\ref{thm:satlwe}, and gives their consequences for lattice problems (Corollaries~\ref{cor:satusvp} and~\ref{cor:satregev}).

\subsection{Hardness from worst-case lattice problems}\label{sec:satbdd}

To solve BDD, we make the lattice integral (Lemma~\ref{lem:integralize}), we try error bounds that are powers of two, and we find the closest vector coordinate by coordinate (Theorem~\ref{thm:satbdd}).
Combined with reductions of Lyubashevsky and Micciancio, this gives consequences for uSVP and GapSVP (Corollary~\ref{cor:satusvp}).

To make the lattice integral, we scale a basis and round its entries.
The argument is standard, and we include it because we need explicit bounds.

\begin{lemma}\label{lem:integralize}
There is a deterministic polynomial-time algorithm that, given a basis $B\in\Q^{n\times n}$ of a lattice $L$, a target $t\in\Q^n$, and a rational $\nu>\sqrt n$, computes a basis $B'$ of $L$, a vector $v_0\in L$, a rational $S>0$, a nonsingular matrix $M_0\in\Z^{n\times n}$, and $t'\in\Z^n$, of encoding size polynomial in that of $(B,t,\nu)$, with the following properties.
We have $\nu(3/2)^n\le S\lambda_1(L)\le\nu(3/2)^n2^{(n+1)/2}$ and $\lambda_1(M_0\Z^n)\ge(1-\sqrt n/\nu)\,S\lambda_1(L)$.
Moreover, let $\rho=\operatorname{dist}(t,L)$, let $v^*\in L$ with $\|t-v^*\|=\rho$, and let $y^*\in\Z^n$ with $v^*=v_0+B'y^*$.
Then $e^*=t'-M_0y^*$ is an integral vector with $\|e^*\|_\infty\le(1+1/\nu)S\rho+\frac12+\frac12(3/2)^n$.
\end{lemma}
\begin{proof}
We first compute $B'$, $S$, and $v_0$, with the LLL algorithm with parameter $3/4$ \cite{LLL82} and the nearest plane algorithm \cite{Babai86}, as presented in \cite[Chapter~2]{MicciancioGoldwasser02}.
The LLL algorithm computes a basis $B'$ of $L$ with Gram--Schmidt vectors $b_1^\circ,\dots,b_n^\circ$ and coefficients $\mu_{kl}$ with $|\mu_{kl}|\le\frac12$ for $l<k$.
Its first vector $b'_1$ satisfies $\|b'_1\|\le2^{(k-1)/2}\|b_k^\circ\|$ for every $k$.
Let $\tau=\min_k\|b_k^\circ\|$.
Every nonzero vector $B'y$ of $L$ satisfies $\|B'y\|\ge|y_k|\,\|b_k^\circ\|\ge\tau$ for the largest $k$ with $y_k\ne0$, so $\tau\le\lambda_1(L)$.
Hence $\tau\le\lambda_1(L)\le\|b'_1\|\le2^{(n-1)/2}\tau$.
Since $\tau$ may be irrational, we use a power of two in place of $1/\tau$.
We compute the smallest integer $k_0$ with $4^{k_0}\tau^2\ge1$, which is possible since $\tau^2$ is rational, and we let $S=\nu(3/2)^n2^{k_0}$.
Let $S_0=\nu(3/2)^n/\tau$.
Since $2^{k_0}\ge1/\tau>2^{k_0-1}$, we have $S_0\le S<2S_0$.
By the nearest plane algorithm, we compute $v_0\in L$ such that $t_0=t-v_0=\sum_k\beta_kb_k^\circ$ with $|\beta_k|\le\frac12$.

Next we bound the coefficients of lattice vectors in the basis $B'$.
We need this because rounding $SB'$ changes $SB'y$ by a vector of norm at most proportional to $\|y\|_1$.
For $y\in\R^n$, write $B'y=\sum_k\eta_kb_k^\circ$ with $\eta_k=y_k+\sum_{l>k}\mu_{lk}y_l$.
Then $|\eta_k|\le\|B'y\|/\|b_k^\circ\|$.
If $|\eta_l|\le\xi$ for all $l$, the recursion $|y_k|\le|\eta_k|+\frac12\sum_{l>k}|y_l|$ gives, by induction from $k=n$ down, $|y_k|\le\xi(3/2)^{n-k}$, and hence $\|y\|_1\le2\xi(3/2)^n$.
Applied with $\xi=\|B'y\|/\tau$, this gives $\|y\|_1\le2(3/2)^n\|B'y\|/\tau$ for every $y$.
For $y=y^*$, we have $B'y^*=t_0-(t-v^*)$, so $\eta_k=\beta_k-(t-v^*)\T b_k^\circ/\|b_k^\circ\|^2$ and $|\eta_k|\le\frac12+\rho/\tau$.
Hence $\|y^*\|_1\le(1+2\rho/\tau)(3/2)^n$.

Finally, we round, and we prove the bounds of the statement.
Let $M_0=\lfloor SB'\rceil$ and $t'=\lfloor St_0\rceil$, rounding entrywise, and let $E_0=M_0-SB'$, whose entries lie in $[-\frac12,\frac12]$.
For $y\in\Z^n\setminus\{\0\}$, we have $\|E_0y\|\le\frac{\sqrt n}2\|y\|_1\le\sqrt n(3/2)^n\|B'y\|/\tau=(\sqrt n/\nu)S_0\|B'y\|\le(\sqrt n/\nu)S\|B'y\|$.
Hence $\|M_0y\|\ge(1-\sqrt n/\nu)S\|B'y\|\ge(1-\sqrt n/\nu)S\lambda_1(L)>0$.
In particular, $M_0$ is nonsingular, and $\lambda_1(M_0\Z^n)\ge(1-\sqrt n/\nu)S\lambda_1(L)$.
Moreover, $e^*=t'-M_0y^*$ is integral, since $t'$, $M_0$, and $y^*$ are, and $e^*=(t'-St_0)+S(t-v^*)-E_0y^*$, so $\|e^*\|_\infty\le\frac12+S\rho+\frac12\|y^*\|_1\le\frac12+S\rho+\frac12(3/2)^n+(3/2)^n\rho/\tau$, which gives the stated bound, since $(3/2)^n\rho/\tau=S_0\rho/\nu\le S\rho/\nu$.
The bounds on $S\lambda_1(L)$ follow from $S_0\tau=\nu(3/2)^n$, $S_0\le S<2S_0$, and $\tau\le\lambda_1(L)\le2^{(n-1)/2}\tau$.
The algorithm runs in polynomial time, since the LLL and nearest plane algorithms do, and the encoding sizes are polynomial, since those of $B'$, $v_0$, and $\tau^2$, and hence of $k_0$, $S$, $M_0$, and $t'$, are.
\end{proof}

\begin{theorem}\label{thm:satbdd}
BDD$_{1/(4n^{5/2})}$ can be solved in deterministic polynomial time with an oracle for satisfiability under the promise that the query is a thin formula $F(p,h,H,D,t)$.
Every query is a formula $F(p,h,H,D,t)$ of Proposition~\ref{prop:satlattice}.
\end{theorem}

\begin{proof}
Let $B\in\Q^{n\times n}$ be a basis of a lattice $L$ and let $t\in\Q^n$, with total encoding size $\sigma$.
If $n<150$, we compute a vector of $L$ closest to $t$ with Kannan's algorithm \cite{Kannan87}, in time polynomial in $\sigma$, since $n$ is bounded by a constant.
So let $n\ge150$.

We use the following parameters.
Let $h_{\max}=h_{\max}(n)$ be the smallest integer with $2^{h_{\max}}\ge8n^4(3/2)^n2^{(n+1)/2}+4(3/2)^n+8$.
Then $h_{\max}\le n\log_2(3/\sqrt2)+4\log_2n+6$, where $\log_2(3/\sqrt2)<1.09$, and hence $3h_{\max}\le4n$, since $n\ge150$.
It follows that, since $2h_{\max}/n^{5/2}\le\frac83n^{-3/2}$,
\begin{equation}\label{eq:hmax}
\tfrac23(1+n^{-4})+2h_{\max}/n^{5/2}\le1-n^{-7/2}.
\end{equation}
We have $h_{\max}\ge7$.
For each prime at most $h_{\max}$, we take the largest power of this prime that is at most $h_{\max}$, and we let $p$ be the vector of these prime powers.
For instance, for $h_{\max}=10$, they would be $8=2^3$, $9=3^2$, $5$, and $7$.
These moduli are pairwise coprime, and their product $\Pi$ is $\operatorname{lcm}(1,\dots,h_{\max})\ge2^{h_{\max}}$, by Nair's inequality \cite{Nair82}.
Moreover, there are at most $h_{\max}$ moduli, each at most $h_{\max}$, so $p_{\max}\le h_{\max}$ and $\sum_jp_j\le h_{\max}^2$.

Let $\alpha_0=1/(3n^{3/2}h_{\max})$.
Since $3h_{\max}\le4n$, we have $1/(4n^{5/2})\le\alpha_0$, so it suffices to solve BDD$_{\alpha_0}$.
As in the proof of Theorem~\ref{thm:bdd}, we try the error bounds $2^{h-1}$ for $h=1,\dots,h_{\max}$, in increasing order.
The first formula reported satisfiable has a unique satisfying assignment, which determines the closest vector, and we find this vector bit by bit, with queries for thin formulas defined by sublattices.

We first make the lattice integral and construct the formulas.
Apply Lemma~\ref{lem:integralize} with $\nu=n^4$, and let $M=M_0\Z^n$ and $D=|\det M_0|$.
Then $\lambda_1(M)\ge(1-\sqrt n/\nu)\nu(3/2)^n>1$, so $M\ne\Z^n$ and $D\ge2$.
Let $H\in\{0,\dots,D-1\}^{n\times n}$ with $H\equiv\operatorname{adj}(M_0)\pmod D$.
Then $M=\Lambda_D^\perp(H)$, since $M_0^{-1}=\operatorname{adj}(M_0)/\det M_0$ and $|\det M_0|=D$, so $y\in M_0\Z^n$ if and only if $\operatorname{adj}(M_0)y\equiv\0\pmod D$.
For $h=1,\dots,h_{\max}$, let $F_h=F(p,h,H,D,t')$, with $m=m'=n$, which is possible by Proposition~\ref{prop:satlattice}, since the product $\Pi$ of the moduli satisfies $\Pi\ge2^{h_{\max}}\ge2^h$.
By Proposition~\ref{prop:satlattice}(a), the satisfying assignments of $F_h$ correspond to the vectors $\tilde y\in M$ with $t'-\tilde y\in\mathcal B_h$.

Next we describe the queries and bound the running time.
We ask the oracle whether $F_1,\allowbreak F_2,\allowbreak\dots,\allowbreak F_{h_{\max}}$ are satisfiable, until it reports that some $F_h$ is.
If it reports none, we output $v_0$.
Otherwise, for $i=1,\dots,n$, we compute $I_i=\{y_i: y\in\R^n,\ t'-M_0y\in[-(2^{h-1}-1),2^{h-1}]^n\}$ by linear programming, we let $c_i=\lceil\min I_i\rceil$, and we let $K_i$ be the smallest nonnegative integer with $2^{K_i}>\max I_i-c_i$.
For $k=1,\dots,K_i$, let $S_{i,k}$ be the diagonal matrix with $i$th diagonal entry $2^k$ and all other diagonal entries $1$, let $D_{i,k}=2^kD$, and let $H_{i,k}\in\{0,\dots,D_{i,k}-1\}^{n\times n}$ with $H_{i,k}\equiv\operatorname{adj}(M_0S_{i,k})\pmod{D_{i,k}}$, so that $M_0S_{i,k}\Z^n=\Lambda^\perp_{D_{i,k}}(H_{i,k})$, as for $M$, since $|\det(M_0S_{i,k})|=D_{i,k}$.
Starting from $u_{i,0}=0$, for $k=1,\dots,K_i$ in turn, we ask the oracle whether $F(p,h,H_{i,k},D_{i,k},t'-(c_i+u_{i,k-1})M_0e_i)$ is satisfiable, and we let $u_{i,k}=u_{i,k-1}$ if it reports that it is and $u_{i,k}=u_{i,k-1}+2^{k-1}$ otherwise.
We let $z_i=c_i+u_{i,K_i}$, and we output $v_0+B'z$.
Every query is a formula of Proposition~\ref{prop:satlattice}, namely some $F_h$ or some $F(p,h,H_{i,k},D_{i,k},t'-(c_i+u_{i,k-1})M_0e_i)$.
By Lemma~\ref{lem:integralize}, $\log D$ and the encoding sizes of $M_0$ and $t'$ are polynomial in $\sigma$.
The endpoints of $I_i$ are optimal values of linear programs, so they have encoding size polynomial in $\sigma$ \cite[Corollary~10.2a]{Schrijver86}, and so $K_i$ and $\log D_{i,k}=\log D+k$ are polynomial in $\sigma$.
Since also $\sum_jp_j\le h_{\max}^2$, every query can be computed in time polynomial in $\sigma$, by Proposition~\ref{prop:satlattice}.
Hence, on every input, the oracle is called at most $h_{\max}+\sum_iK_i$ times, on queries of encoding size polynomial in $\sigma$, and all other computations take time polynomial in $\sigma$.

We now show that the output is $v^*$ if the input satisfies the promise $\rho\le\alpha_0\lambda_1(L)$, where $\rho$, $v^*$, $y^*$, and $e^*$ are as in Lemma~\ref{lem:integralize}.
Let $h^*$ be the smallest $h\ge1$ with $2^{h-1}>\|e^*\|_\infty$.
Then $\|e^*\|_\infty\le2^{h^*-1}-1$, since $e^*$ is integral, and so $e^*\in\mathcal B_{h^*}$.
If $h^*=1$, then $e^*=\0$.
If $h^*\ge2$, then $2^{h^*-2}\le\|e^*\|_\infty$, and so $2^{h^*}-1<4\|e^*\|_\infty$.

We start by showing that $h^*\le h_{\max}$.
By Lemma~\ref{lem:integralize}, $\|e^*\|_\infty\le2S\rho+1+(3/2)^n$, and, by the promise, $S\rho<S\lambda_1(L)\le\nu(3/2)^n2^{(n+1)/2}$.
Hence $2^{h^*}\le4\|e^*\|_\infty+2<8n^4(3/2)^n2^{(n+1)/2}+4(3/2)^n+6<2^{h_{\max}}$.

Next we show that all queries made for $F_h$ with $h\le h^*$ satisfy the promise.
We claim that $F_{h^*}$ is thin, that is, $\lambda_1(M)>n^{3/2}p_{\max}(2^{h^*}-1)/2$.
By Lemma~\ref{lem:integralize}, $\lambda_1(M)\ge(1-\sqrt n/\nu)S\lambda_1(L)$ and $S\lambda_1(L)\ge\nu(3/2)^n$.
If $h^*=1$, the claim holds since $\lambda_1(M)\ge\frac12n^4(3/2)^n>\frac12n^{3/2}h_{\max}\ge\frac12n^{3/2}p_{\max}$.
If $h^*\ge2$, then
\begin{align*}
\frac{n^{3/2}p_{\max}(2^{h^*}-1)}2&<2n^{3/2}h_{\max}\|e^*\|_\infty\le2n^{3/2}h_{\max}\Big((1+\tfrac1\nu)S\rho+\tfrac12+\tfrac12(3/2)^n\Big)\\
&\le\Big(\tfrac23(1+\tfrac1\nu)+\tfrac{2n^{3/2}h_{\max}}{\nu}\Big)S\lambda_1(L),
\end{align*}
using $S\rho\le S\lambda_1(L)/(3n^{3/2}h_{\max})$ and $1+(3/2)^n\le2(3/2)^n\le\frac2\nu S\lambda_1(L)$.
Since $\nu=n^4$, the right-hand side is at most $(1-\sqrt n/\nu)S\lambda_1(L)\le\lambda_1(M)$, by \eqref{eq:hmax}.
Hence, for every $h\le h^*$, the formula $F_h$ is thin, since $\lambda_1(M)>n^{3/2}p_{\max}(2^h-1)/2$.
So are the formulas queried after $F_h$, which have the same $p$ and $h$, since their lattices $M_0S_{i,k}\Z^n$ are contained in $M$, and so $\lambda_1(M_0S_{i,k}\Z^n)\ge\lambda_1(M)$.
Hence the oracle answers all queries made for $F_h$ correctly.

Finally, we show that the procedure outputs $v^*$.
By Proposition~\ref{prop:satlattice}(a), the satisfying assignments of $F_{h^*}$ correspond to the vectors $\tilde y\in M$ with $t'-\tilde y\in\mathcal B_{h^*}$.
One of them is $M_0y^*$, since $t'-M_0y^*=e^*\in\mathcal B_{h^*}$, and so, by Proposition~\ref{prop:satlattice}(b), it is the only one.
For $h\le h^*$, we have $\mathcal B_h\subseteq\mathcal B_{h^*}$, so the satisfying assignments of $F_h$ correspond to vectors among these, that is, to $M_0y^*$ at most.
Thus, for $h\le h^*$, the formula $F_h$ has at most one satisfying assignment, and $F_{h^*}$ has one.
Hence the procedure stops at some $h\le h^*$, and $M_0y^*$ is the only vector $\tilde y\in M$ with $t'-\tilde y\in\mathcal B_h$.
Fix $i$ and $k$, and let $u=u_{i,k-1}$.
By Proposition~\ref{prop:satlattice}(a), the satisfying assignments of $F(p,h,H_{i,k},D_{i,k},t'-(c_i+u)M_0e_i)$ correspond to the vectors $M_0S_{i,k}x$ with $x\in\Z^n$ and $t'-(c_i+u)M_0e_i-M_0S_{i,k}x\in\mathcal B_h$.
Since $M_0S_{i,k}x+(c_i+u)M_0e_i=M_0y$ for $y_i=2^kx_i+c_i+u$ and $y_l=x_l$ for $l\ne i$, these correspond to the vectors $M_0y$ with $y\in\Z^n$, $y_i\equiv c_i+u\pmod{2^k}$, and $t'-M_0y\in\mathcal B_h$.
So this formula is satisfiable if and only if $y^*_i\equiv c_i+u\pmod{2^k}$.
Since $t'-M_0y^*\in\mathcal B_h$, $y^*_i$ is an integer in $I_i$, so $0\le y^*_i-c_i<2^{K_i}$, and induction on $k$ gives $u_{i,k}=(y^*_i-c_i)\bmod2^k$, so $z_i=y^*_i$.
Hence $z=y^*$, and the output is $v_0+B'y^*=v^*$.
\end{proof}

\begin{proof}[Proof of Theorem~C]
By Proposition~\ref{prop:satlattice}, every formula $F(p,h,H,D,t)$ has clauses with at most three literals, and, if it is thin, it satisfies $P(F)'=P(F)_I$.
By Theorem~\ref{thm:satbdd}, every query is such a formula.
Hence Theorem~\ref{thm:satbdd} gives Theorem~C.
\end{proof}

Replacing the oracle by an algorithm and combining Theorem~\ref{thm:satbdd} with Theorem~\ref{thm:lm}, we obtain the further consequences stated in the introduction.

\begin{corollary}\label{cor:satusvp}
Suppose that an algorithm solves satisfiability in polynomial time for formulas $F$ whose clauses have at most three literals, under the promise that every rational subpolyhedron of $P(F)$, and in particular $P(F)$ itself, has Chv\'atal rank at most one.
Then BDD$_{1/(4n^{5/2})}$ and uSVP$_{4n^{5/2}}$ can be solved in deterministic polynomial time, and, for $n\ge2$, GapSVP$_\gamma$ with $\gamma=4n^3/\sqrt{\log n}$ can be solved in randomized polynomial time.
If the algorithm is randomized, then these statements hold with randomized algorithms for the three problems.
\end{corollary}
\begin{proof}
By Proposition~\ref{prop:satlattice}(b), every thin formula satisfies this promise, so, as in the proof of Theorem~C, the algorithm can serve as the oracle in Theorem~\ref{thm:satbdd}.
The statements then follow as in the proof of Corollary~\ref{cor:usvp}, using Theorem~\ref{thm:lm} with $\gamma=4n^{5/2}$ and $\gamma=4n^3/\sqrt{\log n}$, for which the BDD factor is $1/(4n^{5/2})$.
\end{proof}

The factor $1/(4n^{5/2})$ is worse than the factor $1/(4n^2)$ of Theorem~\ref{thm:bdd} by a factor $\sqrt n$.
The loss comes from the moduli: the errors after Lemma~\ref{lem:integralize} can be as large as $2^{O(n)}$, so the largest modulus must be at least of order $n$, which costs a factor of order $n$ in the condition of Proposition~\ref{prop:satlattice}(b), partly offset by a factor of order $\sqrt n$ gained in Lemma~\ref{lem:dualcover}, which, unlike Lemma~\ref{lem:dualbasis}, does not need a basis of the dual lattice.
Moreover, unlike that of Theorem~\ref{thm:bdd}, the reduction of Theorem~\ref{thm:satbdd} does not preserve the dimension: the number of variables of its formulas depends on the encoding size of the input, through the determinant of the integral lattice given by Lemma~\ref{lem:integralize}, so Corollary~\ref{cor:satusvp} has no analog of the statement of Corollary~\ref{cor:usvp} on running times $2^{o(n)}$.

\subsection{Hardness from learning with errors}\label{sec:satlwe}

To solve decision-LWE, we use a single formula defined by the $q$-ary lattice $\Lambda_q(A\T)$.
This lattice is integral and the errors are polynomially bounded, so we apply Proposition~\ref{prop:satlattice} directly, with moduli of logarithmic size.
With high probability, this formula is thin, and so satisfies the promise, and it is satisfiable for LWE instances but not for uniform instances.
Thus a single call to an algorithm for satisfiability distinguishes the two cases (Theorem~\ref{thm:satlwe}).
Combined with a result of Regev, this gives consequences for GapSVP and SIVP for quantum algorithms (Corollary~\ref{cor:satregev}).
In this subsection, $n$ is the dimension of $s$, and also of the lattice problems in Theorem~\ref{thm:regev}, while the lattice $\Lambda_q(A\T)$ has dimension $m$.

To apply Proposition~\ref{prop:satlattice}, we write $\Lambda_q(A\T)$ in the form $\Lambda_q^\perp(H)$.

\begin{lemma}\label{lem:kernel}
Let $q$ be a prime, let $n<m$, and let $A\in\Z_q^{n\times m}$.
In time polynomial in $m$ and $\log q$, we can compute a matrix $H\in\{0,\dots,q-1\}^{m'\times m}$ with $1\le m'\le m$ and $\Lambda_q^\perp(H)=\Lambda_q(A\T)$.
\end{lemma}
\begin{proof}
Since $q$ is prime, we can compute by Gaussian elimination over $\Z_q$ a matrix $H\in\{0,\dots,q-1\}^{m'\times m}$ whose rows form a basis of the kernel of $A$ over $\Z_q$.
This kernel has dimension at least $m-n\ge1$, so $1\le m'\le m$.
Over the field $\Z_q$, the orthogonal complement of this kernel is the row space of $A$.
Hence, for $y\in\Z^m$, we have $Hy\equiv\0\pmod q$ if and only if $y\equiv A\T s\pmod q$ for some $s\in\Z_q^n$, that is, $\Lambda_q^\perp(H)=\Lambda_q(A\T)$.
\end{proof}

For the error distribution $\bar\Psi_\alpha$ and suitable parameters, the formulas of Proposition~\ref{prop:satlattice} give an explicit distribution of hard instances.

\begin{theorem}\label{thm:satlwe}
Let $n$ be sufficiently large, let $q=q(n)$ be a prime in $[n^3,2n^3]$, which exists by Bertrand's postulate, let $k=\lceil\log_2q\rceil$, let $\alpha=1/(200n^{3/2}k^{5/2}\log n)$, and let $m=3nk$.
Let $h$ be the largest integer with $h2^h\le q/(4m^{3/2})$.
For each prime at most $h$, take the largest power of this prime that is at most $h$, and let $p$ be the vector of these prime powers.
For $A\in\Z_q^{n\times m}$ and $t\in\Z_q^m$, let $F(A,t)=F(p,h,H,q,t)$, where $H$ is the matrix given by Lemma~\ref{lem:kernel} for $A$, and let $\mathcal F$ be the distribution of $F(A,t)$ when $(A,t)$ is an LWE instance with $\chi=\bar\Psi_\alpha$ or a uniform instance, each with probability $1/2$.
The distribution $\mathcal F$ can be sampled in polynomial time, its samples are formulas with $O(n^2\log^4n)$ variables whose clauses have at most three literals, and a sample of $\mathcal F$ is thin with probability at least $1-2^{-2nk}$.
Suppose that a randomized polynomial-time algorithm decides satisfiability correctly with probability at least $1/2+n^{-c}$ for some constant $c>0$, over the choice of $F$ from $\mathcal F$ and its own randomness.
Then there is a randomized polynomial-time algorithm that solves decision-$\LWE_{n,q,\bar\Psi_\alpha}$ with $m$ samples.
\end{theorem}
\begin{proof}
Let $M=\Lambda_q(A\T)$, so that $M=\Lambda_q^\perp(H)$ by Lemma~\ref{lem:kernel}.

We first show that $F(A,t)$, and hence a sample of $\mathcal F$, can be computed in time polynomial in $n$.
Since $n$ is sufficiently large, we have $q/(4m^{3/2})\ge n^3/(12\sqrt3\,n^{3/2}k^{3/2})\ge7\cdot2^7$, so $h\ge7$, and $h\le\log_2q\le k$.
The product $\Pi$ of the moduli is $\operatorname{lcm}(1,\dots,h)\ge2^h$, by Nair's inequality \cite{Nair82}.
Moreover, there are at most $h$ moduli, each at most $h$, so $p_{\max}\le h$, $\sum_jp_j\le h^2$, and $\Pi\le h^h$.
Hence, by Proposition~\ref{prop:satlattice}, $F(A,t)$ is a formula whose clauses have at most three literals, and it can be computed in time polynomial in $n$, and so can a sample of $\mathcal F$, since $q$ can be found in time polynomial in $n$ by testing the integers in $[n^3,2n^3]$ for primality, and $\bar\Psi_\alpha$ can be sampled in polynomial time.

The formulas $F(A,t)$ have $O(n^2\log^4n)$ variables.
Indeed, by Proposition~\ref{prop:satlattice}(c), since $h\le k$, $\sum_jp_j\le h^2$, $\log_2\Pi\le h\log_2h$, $m'\le m$, and $D=q$, they have $O(mh^3\log h+m^2h\log q)=O(n^2k^4)$ variables.

Next we show that $F(A,t)$ is thin with high probability.
In LWE instances and in uniform instances alike, $A$ is uniformly random, so $\lambda_1(M)>q/8$ with probability at least $1-q^n2^{-m}\ge1-2^{-2nk}$, by Lemma~\ref{lem:lambda1} and $q^n\le2^{nk}=2^{m/3}$.
If $\lambda_1(M)>q/8$, then, since $p_{\max}\le h$ and $h2^h\le q/(4m^{3/2})$,
\[
\frac{m^{3/2}p_{\max}(2^h-1)}2<\frac{m^{3/2}h2^h}2\le\frac q8<\lambda_1(M),
\]
and so $F(A,t)$ is thin.

We then show that, for an LWE instance, $F(A,t)$ is satisfiable with probability at least $1-2m\exp(-\pi\log^2n)$.
By Proposition~\ref{prop:satlattice}(a), the satisfying assignments of $F(A,t)$ correspond to the vectors $\tilde y\in M$ with $t-\tilde y\in\mathcal B_h$.
Let $(A,t)$ be an LWE instance with error $e$, and let $\varepsilon_\chi$ be the probability that $\|e\|_\infty>2^{h-1}-1$.
Then $t-e\in M$, so $F(A,t)$ is satisfiable if $\|e\|_\infty\le2^{h-1}-1$, and hence with probability at least $1-\varepsilon_\chi$.
By the choice of $h$, we have $(h+1)2^{h+1}>q/(4m^{3/2})$, and so, since $h+1\le2k$ and $16m^{3/2}=48\sqrt3\,n^{3/2}k^{3/2}$,
\[
2^{h-1}>\frac q{16m^{3/2}(h+1)}\ge\frac q{96\sqrt3\,n^{3/2}k^{5/2}}.
\]
Since $q\ge n^3$ and $n$ is sufficiently large, it follows that $2^{h-1}-\frac12\ge q/(200n^{3/2}k^{5/2})=\alpha q\log n$.
As in the proof of Theorem~\ref{cor:lwe}, each entry of $e$ has absolute value greater than $2^{h-1}-1$ with probability at most $\Pr[|X|\ge\alpha q\log n]\le2\exp(-\pi\log^2n)$, where $X$ is normal with mean zero and standard deviation $\alpha q/\sqrt{2\pi}$.
Hence, by the union bound over the $m$ entries, $\varepsilon_\chi\le2m\exp(-\pi\log^2n)$.

We also show that, for a uniform instance, $F(A,t)$ is satisfiable with probability at most $2^{-m}$.
Let $(A,t)$ be a uniform instance.
If $F(A,t)$ is satisfiable, then $t\equiv A\T s+e'\pmod q$ for some $s\in\Z_q^n$ and $e'\in\mathcal B_h$.
There are at most $q^n2^{hm}$ such residues, so $F(A,t)$ is satisfiable with probability at most $q^n(2^h/q)^m$.
Since $2^h\le q/(4m^{3/2}h)\le q/4$ and $q^n\le2^{m/3}$, this probability is at most $2^{m/3}4^{-m}\le2^{-m}$.

It remains to construct the algorithm for decision-LWE.
Given $(A,t)$, our algorithm computes $F(A,t)$, runs the given algorithm on it, and outputs $1$ if and only if the given algorithm reports that $F(A,t)$ is satisfiable.
As in the proof of Theorem~\ref{cor:lwe}, the probabilities that it outputs $1$ on LWE instances and on uniform instances differ by at least $2n^{-c}-\varepsilon_\chi-2^{-m}$, which is at least $n^{-c}$ for sufficiently large $n$.
\end{proof}

\begin{proof}[Proof of Theorem~D]
By Proposition~\ref{prop:satlattice}(b), the SAT polytopes of thin formulas satisfy $P(F)'=P(F)_I$.
So the SAT polytope of a sample of $\mathcal F$ satisfies $P(F)'=P(F)_I$ with high probability.
Moreover, $2/3\ge1/2+1/n$ for $n\ge6$.
Hence Theorem~\ref{thm:satlwe} gives Theorem~D.
\end{proof}

Combining Theorem~\ref{thm:satlwe} with Theorem~\ref{thm:regev}, we obtain the consequence for quantum algorithms stated in the introduction.

\begin{corollary}\label{cor:satregev}
Suppose that there is a randomized polynomial-time algorithm as in Theorem~\ref{thm:satlwe}.
Then, for some $\gamma=\tilde O(n^{5/2})$, there is a polynomial-time quantum algorithm that solves GapSVP$_\gamma$ and SIVP$_\gamma$ on $n$-dimensional lattices.
\end{corollary}
\begin{proof}
Let $q$, $k$, $\alpha$, and $m$ be as in Theorem~\ref{thm:satlwe}.
By Theorem~\ref{thm:satlwe}, there is a randomized polynomial-time algorithm that solves decision-$\LWE_{n,q,\bar\Psi_\alpha}$ with $m$ samples.
Theorem~\ref{thm:regev} applies, since $q\le2n^3$ is prime, $q$ and $m=3nk$ are bounded by a polynomial in $n$, and $\alpha q\ge n^3/(200n^{3/2}k^{5/2}\log n)>2\sqrt n$ for sufficiently large $n$, and it gives $\gamma=\tilde O(n/\alpha)=\tilde O(n^{5/2})$.
\end{proof}

\subsection{Refutations and set partitioning polytopes}\label{sec:satremarks}

We conclude with two remarks, on refutations and on set partitioning polytopes.

\begin{remark}\label{rem:findw}
As with thin parallelepipeds in Remark~\ref{rem:refutation}, the SAT polytopes of thin formulas $F=F(p,h,H,D,t)$ and their rational subpolyhedra that contain no integral vector have empty Chv\'atal closure, by Proposition~\ref{prop:satlattice}(b), and so, by Helly's theorem, refutations with at most $N+1$ CG cuts, where $N$ is the number of variables of $F$.
The vectors $w$ of the proof of Proposition~\ref{thm:crt} play the role of the short dual basis, and both these vectors and refutations are hard to find.

We now show that the vectors $w$ and refutations are hard to find.
Given the vectors $w$, we compute $z$ and $z'$ with $Dw=H\T z+Dz'$, as in the proof of Proposition~\ref{thm:crt}, by solving a system of linear Diophantine equations \cite[Corollary~5.3b]{Schrijver86}, and then the CG cuts of that proof by linear programming.
As in that proof, these cuts decide whether the functional formulation from which $F$ is obtained has an integral vector.
By Lemma~\ref{lem:exactone}, this decides satisfiability of $F$.
So, by Theorems~\ref{thm:satbdd} and~\ref{thm:satlwe}, finding the vectors $w$ is as hard as BDD$_{1/(4n^{5/2})}$ and, since the formulas of Theorem~\ref{thm:satlwe} are thin with high probability, as hard as decision-LWE.
The same holds for finding any refutation, since, as in Remark~\ref{rem:refutation}, an algorithm that finds a refutation whenever one exists decides satisfiability for thin formulas.
\end{remark}

\begin{remark}\label{rem:setpart}
Theorems~C and~D also hold for integer feasibility on every class of polytopes that contains the SAT polytopes of formulas whose clauses have at most three literals, for instance the polytopes in $[0,1]^N$ defined by inequalities with coefficients in $\{0,\pm1\}$, integral right-hand sides, and at most three nonzero coefficients each.
They also hold for \emph{set partitioning polytopes} $\{x\ge\0: Wx=\1\}$, where $W$ is a $0/1$ matrix, with at most three ones in each row, although this class does not contain these SAT polytopes.
Without the promise, integer feasibility on these set partitioning polytopes is NP-complete, even with exactly three ones in each row \cite{Schaefer78}.

We now prove the statement for set partitioning polytopes.
Let $F$ be one of the formulas of Proposition~\ref{prop:satlattice}, and let $X(\mathcal T_3)$ be the face of $P(F)$ from the proof of Lemma~\ref{lem:exactone}.
Here $\mathcal T_3$ is a family of sets of at most three literals, and $X(\mathcal T_3)$ contains all integral vectors of $P(F)$.
Adding a variable $\bar x_i$ and the equation $x_i+\bar x_i=1$ for every $i$, and writing every literal $\bar x_i$ as this variable, transforms $X(\mathcal T_3)$ into a set partitioning polytope $\{x\ge\0: Wx=\1\}$ whose matrix $W$ has at most three ones in each row.
This transformation is a map as in Lemma~\ref{lem:facelift}, so the integral vectors correspond one-to-one.
Every rational subpolyhedron of $X(\mathcal T_3)$ is one of $P(F)$, so, by Lemma~\ref{lem:facelift}, every rational subpolyhedron of the set partitioning polytope has Chv\'atal rank at most one whenever every rational subpolyhedron of $P(F)$ does.
Every query of Theorem~\ref{thm:satbdd} and every formula $F(A,t)$ of Theorem~\ref{thm:satlwe} is a formula of Proposition~\ref{prop:satlattice}, so this applies to all of them.
\end{remark}

\bigskip
\begin{small}\sloppy
\paragraph{Funding.} A. Del Pia is partially supported by AFOSR Grant \mbox{FA9550-23-1-0433} and ONR Grant \mbox{N00014-25-1-2490}. Any opinions, findings, conclusions, or recommendations expressed in this material are those of the author and do not necessarily reflect the views of the Air Force Office of Scientific Research or the Office of Naval Research.

\paragraph{AI declaration.} AI tools (Claude, developed by Anthropic) were used for proofreading and suggesting improvements to the exposition. The author takes full responsibility for the content of the paper.
\end{small}

{\small
\bibliographystyle{plain}
\bibliography{refs}
}

\end{document}